\documentclass[a4paper,reqno,11pt]{amsart}
\usepackage{amsfonts}

\usepackage{geometry} 
\usepackage{amsmath,amsthm,amssymb}                
\usepackage{graphicx, color}
\usepackage{amssymb}
\usepackage{epstopdf}
\usepackage{pdfsync}
\usepackage{mathabx}

\newtheorem{definition}{Definition}[section]
\newtheorem{lemma}[definition]{Lemma}
\newtheorem{theorem}[definition]{Theorem}
\newtheorem{proposition}[definition]{Proposition}
\newtheorem{corollary}[definition]{Corollary}
\newtheorem{remark}[definition]{Remark}

\numberwithin{equation}{section}

\def\e{\varepsilon}

\newcommand{\weak}{\rightharpoonup}
\newcommand{\weakstar}{\buildrel * \over \rightharpoonup}
\def\R{\mathbb R}
\def\Q{\mathbb Q}
\def\N{\mathbb N}
\def\HH{\mathcal H}\def\AA{\mathcal A}\def\DD{\mathcal D}
\def\tf{\widetilde f}

\newcommand{\LM}[2]{\hbox{\vrule width.4pt \vbox to#1pt{\vfill
\hrule width#2pt height.4pt}}}
\newcommand{\LLL}{{\mathchoice {\>\LM{7}{5}\>}{\>\LM{7}{5}\>}{\,\LM{5}{3.5}\,}{\,\LM{3.35}{2.5}\,}}}

\def\dxy{\,dx dy}
\def\dx{\,dx}
\def\dy{\,dy}

\title[Compactness for a class of integral functionals]{Compactness for a class of integral functionals\\
 with interacting local and non-local terms}
\author[A. Braides]{Andrea Braides}
\address[Andrea Braides]{SISSA, Via Bonomea 265, 34136 Trieste, Italy}
\email[Andrea Braides]{abraides@sissa.it}
\author[G. Dal Maso]{Gianni Dal Maso}
\address[Gianni Dal Maso]{SISSA, Via Bonomea 265, 34136 Trieste, Italy}
\email[Gianni Dal Maso]{dalmaso@sissa.it}

\thanks{Preprint SISSA 21/2022/MATE} 

\begin{document}

\maketitle

\begin{abstract} We prove a compactness result  with respect to $\Gamma$-convergence for a class of integral functionals which are expressed as a sum of a local and a non-local term. The main feature is that, under our hypotheses, the local part of the $\Gamma$-limit depends on the interaction between the local and non-local terms of the converging subsequence. The result is applied to concentration and homogenization problems.
\end{abstract}

\section{Introduction} In this paper we study sequences  of integral functionals of the form
\begin{equation}\label{effe-k-intro}
\Phi_k(u):=\int_{\Omega\times\Omega} f_k(x,y,u(x),u(y)) \dxy+\int_\Omega g_k(x,\nabla u(x))\,dx
\end{equation}
defined on $W^{1,p}(\Omega)$, where $\Omega$ is an open bounded subset of $\mathbb R^d$ with Lipschitz boundary and $1<p<\infty$. Our scope is to show that, under proper structure properties on $f_k$ and $g_k$, such sequences are precompact with respect to the $\Gamma$-convergence in the weak topology of $W^{1,p}(\Omega)$, and their $\Gamma$-limit can be written in an integral form as
\begin{equation}\label{effe-intro}
\Phi(u):=\int_{\Omega\times\Omega} f(x,y,u(x),u(y)) \dxy+\int_\Omega g(x,\nabla u(x))\,dx.
\end{equation}

Our hypotheses (see Section \ref{Hypotheses}) allow for concentration of $f_k$ close to the diagonal $\Delta=\{(x,y): x=y\}$ (see Remark \ref{concentration}). As a result of this, while the non-local term
\begin{equation}\label{non-f}
\int_{\Omega\times\Omega} f(x,y,u(x),u(y)) \dxy
\end{equation}
of the $\Gamma$-limit depends only on the sequence $\{f_k\}$, the local term
\begin{equation}\label{g}
\int_\Omega g(x,\nabla u(x))\,dx
\end{equation}
depends on the mutual interaction between the local part of $\Phi_k$ and its non-local part.
More precisely, the function $g$ is determined by the sequence $\{g_k\}$ and the
part of the sequence $\{f_k\}$ that is concentrating on the diagonal. 

Non-local functionals as in \eqref{non-f}, which appear naturally in the definition of fractional Sobolev spaces,  have been recently studied also in connection with variational models in peridynamics \cite{E,S}, image processing \cite{BN,GO}, and data analysis \cite{SOZ,ST}. Semicontinuity and relaxation problems with respect to the weak topology of $L^p(\Omega)$ have also been studied  e.g.~in \cite{BMC,KZ,MCT,P}.

Sequences of non-local functionals as in \eqref{non-f} (that is, of the form \eqref{effe-k-intro} with $g_k=0$) converging to local functionals (that is, of the form \eqref{g}) have been considered by various authors in different contexts. We refer to \cite{BBM} for their use in the characterization of Sobolev spaces by approximation (see also \cite{PV}), to \cite{MC,MD} for the analysis of limit problems in peridynamics as the horizon tends to zero, and to \cite{AABPT} for a theory of convolution-type energies.

In the present paper we consider general functionals $\Phi_k$ of the form \eqref{effe-k-intro}, and our hypotheses ensure that the local terms involving $g_k$ are always non-trivial and provide an equi-coerciveness property in the weak topology of $W^{1,p}(\Omega)$. This is the reason why we may study the $\Gamma$-limit in this topology. 

Contrary to many of the papers quoted above, we allow for the possibility that a non-local part is still present  in the limit; that is, our hypotheses are not forcing to have $f=0$ in $\Phi$. 

We focus on a general representation result as in \eqref{effe-intro} for the $\Gamma$-limit $\Phi$, even when no explicit formula for $g$ is available.
While there is an integral-representation theory for $\Gamma$-limits of  sequences of local functionals of the form
$$
\int_\Omega g_k(x,\nabla u(x))\,dx,
$$
 under general hypotheses this cannot be extended to functionals of the complete form \eqref{effe-k-intro}.
However in our case, we are able to separately examine the limit behaviour of the functionals `far from the diagonal' and `close to the diagonal'. 

The off-diagonal part is fully described in terms of the weak$^*$ limits of the sequences of integrands $\{f_k(\cdot,\cdot,s,t)\}$ (see Section \ref{off}). In this argument we use the fact that all integrals are with respect to the Lebesgue measure. For the analysis of a case when the non-local terms are of the form 
\begin{equation}\label{non-f-k}
\int_{\Omega\times\Omega} f_k(x,y,u(x),u(y)) d\mu_k(x,y),
\end{equation}
for suitable measures $\mu_k$ on $\Omega\times\Omega$, we refer to \cite{BDM}.

The analysis of the part concentrating on the diagonal is much more delicate. We introduce a parameter $\delta>0$ and examine the part of the $\Gamma$-limit concentrated on a $\delta$-neighbourhood of the diagonal (Section \ref{diag}). Using some technical estimates we prove that such part, up to a controlled small error as $\delta\to 0$, satisfies all hypotheses of the abstract integral-representation theorem for local functionals (see Section \ref{intrepsec}). 
Finally, these results lead to the desired integral representation \eqref{effe-intro} (see Theorem \ref{thglo}). 

Our abstract compactness result is applied to the study of two prototypical problems in which local and non-local parts interact. In both cases the limit is purely local. The first one (Section \ref{sepsca}) regards the interaction between relaxation in the local part and concentration on the diagonal in the non-local part  of convolution type. We characterize the limit energy density $g$ showing that concentrations and oscillations act at different scales.

 The second example (Section \ref{sechom}) regards the interaction between oscillation and concentration. We consider periodically oscillating energies both in the local and non-local part, with a period of the same scale as the concentration. In this case we have homogenization and the limit $g$ is characterized by an asymptotic homogenization formula in which both local and non-local terms are present. 

\section{Setting of the problem}	

\subsection{Notation}

In the following $d\ge1$ is the space dimension and $\Omega$ is a bounded open subset of $\mathbb R^d$ with Lipschitz boundary. The family of the open subsets of $\Omega$ is denoted by $\AA$. 
The {\em diagonal} of $\Omega\times \Omega$ is denoted by
$$\Delta:= \{(x,x): x\in \Omega\},$$
while
$$\Delta_\delta:= \{(x,y): x,y\in \Omega: |x-y|<\delta\}$$
denotes a $\delta$-neighbourhood of the diagonal.
The notation $\mathcal M_b(\Omega\times\Omega)$ stands for the space of bounded Radon measures on $\Omega\times\Omega$. The Lebesgue measure in $\mathbb R^k$ is denoted by $\mathcal L^k$ and the $m$-dimensional Hausdorff measure (in any space dimension) is denoted by $\HH^m$. The restriction of a measure $\mu$ to the set $A$ is denoted by $\mu\LLL A$, defined by $\mu\LLL A(B)=\mu(B\cap A)$ for any Borel set $B$.

The characteristic function of the set $A$ is denoted  by ${\bf 1}_A$.
If $A'\subset\subset A$, a cut-off function between $A'$ and $A$ is a function $\varphi\in C^\infty_c(\R^d)$ such that $\varphi=1$ on $A'$, $\varphi=0$ on $\R^d\setminus A$ and $0\le \varphi\le 1$.

\subsection{Hypotheses}\label{Hypotheses}
We fix exponents  $1<p\le q<+\infty$ and let $p',q'$ denote the corresponding dual exponents.
If $p<d$ then we also assume that $q< p^*$, where $p^*$ is the Sobolev exponent of $p$.

We consider sequences of  integrals of the type 
\begin{equation}\label{effe-k}
\Phi_k(u):=\int_{\Omega\times\Omega} f_k(x,y,u(x),u(y)) \dxy+\int_\Omega g_k(x,\nabla u(x))\,dx
\end{equation}
defined on $W^{1,p}(\Omega)$.

\bigskip
\subsubsection{Hypotheses on $f_k$}

\smallskip
The functions $f_k\colon\Omega\times\Omega\times \R\times\R\to [0,+\infty)$ are Borel functions satisfying the following conditions.

\bigskip\noindent
1. {\em Behaviour far from the diagonal $\Delta$}.

\nobreak\noindent
For all $\delta>0$ there exist $k_\delta\in\mathbb N$, $c_\delta>0$, $\alpha_\delta\in L^1(\Omega)$, $\beta_\delta\in L^{q'}(\Omega)$ such that
\begin{eqnarray}
&0\le f_k(x,y,s,t)\le c_\delta (\alpha_\delta(x)+|s|^q)(\alpha_\delta(y)+|t|^q),\label{fkgrowth}\\
&|f_k(x,y,s_1,t)-f_k(x,y,s_2,t)|\le c_\delta(\alpha_\delta(y)+|t|^q)\bigl(\beta_\delta(x)+(|s_1|\vee|s_2|)^{q-1}\bigr)|s_1-s_2|,\label{fkconts}\qquad\\
&|f_k(x,y,s,t_1)-f_k(x,y,s,t_2)|\le c_\delta(\alpha_\delta(x)+|s|^q)\bigl(\beta_\delta(y)+(|t_1|\vee|t_2|)^{q-1}\bigr)|t_1-t_2|,\qquad\label{fkcontt}
\end{eqnarray}
 for all $k\ge k_\delta$, $x,y\in\Omega$ with $|x-y|\ge\delta$, $s,t,s_1,s_2, t_1,t_2\in \R$.

\bigskip\noindent
2. {\em Behaviour close to the diagonal $\Delta$}.

\noindent There exist Borel functions $\tf_k\colon\Omega\times\Omega\times \R\to [0,+\infty)$ and $b_k\in L^1(\Omega\times\Omega)$ such that there exist a constant $ c_\Delta$ and $\alpha\in L^1(\Omega)$ with $\alpha\ge 0$ such that 
\begin{eqnarray} &
\tf_k(x,y,\cdot) \hbox{ is convex, }\label{convextf}
\\ & \displaystyle
{1\over 2c_\Delta} \tf_k(x,y,2(s-t))\le f_k(x,y,s,t)\le \tf_k(x,y,s-t)+ c_\Delta(\alpha(x)+|s|^q)(\alpha(y)+|t|^q),\quad\label{doubleboundtf}
\\ &  \displaystyle
\tf_k(x,y,\tau)\le {b_k(x,y)\over |x-y|^p} |\tau|^p,\label{growthbk}
\end{eqnarray}
for all $k\in\N$, $x,y\in\Omega$, and $s,t,\tau\in \R$. Moreover, we assume the technical hypothesis that there exists a constant $c_b>0$ such that
\begin{equation}\label{cibi}
\int_\Omega b_k(x,y)\dx\le c_b
\end{equation}
for all $y\in\Omega$ and for all $k$. We will see that this requirement can be relaxed to an integrability assumption on the function $y\mapsto \int_\Omega b_k(x,y)\dx$. These properties can be required symmetrically to the function $x\mapsto \int_\Omega b_k(x,y)\dy$.

We suppose that there exist $b_\Delta\in L^\infty(\Delta,\HH^d\LLL\Delta)$ and $b\in L^\infty(\Omega\times\Omega)$ such that
\begin{equation}\label{weakconvbk}
b_k {\mathcal L}^{2d}\weakstar b_\Delta \HH^d\LLL\Delta + b {\mathcal L}^{2d}
\end{equation}
weakly$^*$ in $\mathcal M_b(\Omega\times\Omega)$. 
\bigskip

\begin{remark}[concentration]\label{concentration}\rm Hypotheses \eqref{doubleboundtf}--\eqref{cibi} allow for concentration of $f_k$ on the diagonal in a controlled way, which may occur when $b_\Delta\neq0$. The prototypical example for concentration is when 
$$
f_k(x,y,s,t)= {b_k(x,y)\over |x-y|^p} |s-t|^p, \hbox{ and } b_k(x,y)= k^d\psi(k(x-y)),
$$
where $\psi\in L^1(\mathbb R^d)$ with $\psi\ge 0$. We refer to Section \ref{sepsca} for the treatment of this particular case for a specific choice of the local term.
\end{remark}

Finally, in order to obtain that the local part of the functionals in the limit depend only on the gradient, we will also require an asymptotic invariance by addition of a constant close to the diagonal as follows.

\bigskip\goodbreak\noindent
3. {\em Asymptotic invariance by addition of constants}.

\noindent
For every $r\in\mathbb R$ there exist $\alpha_r\in L^1(\Omega\times\Omega)$ with $\alpha_r\ge 0$ and a modulus of continuity $\omega_r:[0,+\infty)\to[0,+\infty)$ such that
\begin{eqnarray}\label{invariance}\nonumber
&\bigl|f_k(x,y,s+r,t+r)-f_k(x,y,s,t)\bigr|\le \alpha_r(x,y)+ \omega_r(|x-y|)f_k(x,y,s,t)\\
&+ c_\Delta (\alpha(x)+|s|^q)(\alpha(y)+|t|^q)
\end{eqnarray}
for all $k\in\mathbb N$, $x,y\in\Omega$ and $s,t\in\mathbb R$.

Note that this condition is trivially satisfied if $f_k(x,y,s,t)=f_k(x,y,s-t)$. 


\bigskip
\subsubsection{Hypotheses on $g_k$}

\smallskip
The functions $g_k\colon\Omega\times \R^d\to [0,+\infty)$ will be Borel functions satisfying the following condition: there exist constants $c_0,c_1>0$ and $a\in L^1(\Omega)$ such that
\begin{equation}\label{usual}
c_0|\xi|^p\le g_k(x,\xi)\le c_1|\xi|^p+ a(x)
\end{equation}
for all $k\in\N$, $x\in \Omega$, $\xi\in \R^d$.

\subsection{Localized functionals}\label{localiz}
In order to study the $\Gamma$-limit of functionals $\Phi_k$ we introduce a separate notation for the non-local and local parts of the functionals. Furthermore, in order to apply representation techniques for $\Gamma$-limits we consider the restriction of these functionals to open subsets.

\bigskip
As for the nonlocal part, for all $A,B\in\mathcal A$ and $u\in L^q(\Omega)$ we define
\begin{eqnarray}\label{effek}
&\displaystyle F_k(u,A,B):= \int_{A\times B} f_k(x,y,u(x),u(y)) \dxy.
\end{eqnarray}
Furthermore, in order to separately analyze the contributions of the nonlocal part close and far from the diagonal, we set
\begin{eqnarray}\label{effekdelta}
&\displaystyle F^\delta_k(u,A,B):= \int_{(A\times B)\cap \Delta_\delta} f_k(x,y,u(x),u(y)) \dxy,\\ \label{effecekkdelta}
&\displaystyle \widecheck F^\delta_k(u,A,B):= \int_{(A\times B)\setminus\Delta_\delta} f_k(x,y,u(x),u(y)) \dxy.
\end{eqnarray}
Note that 
\begin{equation}\label{trivfk}
F_k(u,A,B)=F^\delta_k(u,A,B)+\widecheck F^\delta_k(u,A,B)
\end{equation}
by definition.

\smallskip
For all $A\in\mathcal A$ and $u\in W^{1,p}(\Omega)$ we set
\begin{eqnarray}\label{defGk}
G_k(u,A):= \int_{A} g_k(x,\nabla u(x))\dx.
\end{eqnarray}

\smallskip We finally extend this notation to the complete functionals as follows. 
For all $A,B\in\mathcal A$ and $u\in W^{1,p}(\Omega)$ we define
\begin{eqnarray}\label{ficappa}
\Phi_k(u,A,B):= F_k(u,A,B)+ G_k(u, A\cap B),\\
\Phi^\delta_k(u,A,B):=F^\delta_k(u,A,B)+ G_k(u, A\cap B).\label{ficappadelta}
\end{eqnarray}
Note that 
\begin{equation}\label{triug}
\Phi_k(u,A,B)= \widecheck F^\delta_k(u,A,B)+\Phi^\delta_k(u,A,B)
\end{equation}
by \eqref{trivfk}.

\bigskip
In order to determine a subsequence on which the $\Gamma$-limit exists,
we now fix  a countable dense subset $\DD$ of $\AA$; that is such that if $A,B\in\AA$ and $A\subset\subset B$ then there exists $D\in\DD$ with $A\subset\subset D \subset\subset B$. By the compactness of $\Gamma$-convergence there exists a subsequence of indices $k$, not relabelled, and a functional $\Phi\colon W^{1,p}(\Omega)\times \DD \times \DD \to[0,+\infty]$ such that for every $A,B\in\DD$ we have 
\begin{equation}\label{gammaD}
\Phi_k(\cdot,A,B)\hbox{ $\Gamma$-converges to }\Phi(\cdot,A,B)
\end{equation}
with respect to the weak topology of $W^{1,p}(\Omega)$. In the rest of the paper we consider this fixed subsequence, from which we will extract further subsequences.

We furthermore define
 \begin{eqnarray}\label{Phim}
 \Phi'(\cdot,A,B)=\Gamma\hbox{-}\liminf_{k\to+\infty} \Phi_k(\cdot,A,B),\\ \label{Phip}
  \Phi''(\cdot,A,B)=\Gamma\hbox{-}\limsup_{k\to+\infty} \Phi_k(\cdot,A,B)
  \end{eqnarray}
  for all $A,B\in\mathcal A$. Note that $ \Phi'(\cdot,A,B)= \Phi''(\cdot,A,B)$ for all $A,B\in\DD$ by \eqref{gammaD}.

\section{Off-diagonal behaviour}\label{off}
We now prove a compactness result for functionals $\widecheck F^\delta_k(u,A,B)$ for the $\Gamma$-convergence with respect to the strong $L^q$-convergence.
We begin with a compactness result for the integrands $f_k$.

\begin{lemma}\label{effelim} There exist a subsequence of $f_k$, not relabelled, and a Borel function $f\colon\Omega\times\Omega\times\R\times\R\to [0,+\infty)$ satisfying \eqref{fkgrowth}--\eqref{fkcontt}, such that for every $\delta>0$ we have
\begin{equation}\label{weakstarfk}
f_k(\cdot,\cdot,s,t)\weak f(\cdot,\cdot,s,t)
\end{equation}
weakly in $L^1((\Omega\times\Omega)\setminus \Delta_\delta)$ for every $s,t\in\R$.\end{lemma}

\begin{proof}
By \eqref{fkgrowth} and by Fubini's theorem, for every $s,t\in\R$ the functions $f_k(\cdot,\cdot,s,t)$ are bounded by a function in $L^1(\Omega\times\Omega)$ independent of $k$. Therefore, using a diagonal argument we can find a subsequence of $f_k$, not relabelled, and a Borel function $f\colon\Omega\times\Omega\times\Q\times\Q\to [0,+\infty)$ such that \eqref{weakstarfk} holds for every $\delta>0$ and for every $s,t\in\Q$. Passing to the limit in \eqref{fkgrowth}--\eqref{fkcontt} we have that $f$ satisfies the same properties for every $s,t\in\Q$. We can then extend $f(x,y,\cdot,\cdot)$ by continuity to $\R\times\R$ for all $x,y\in\Omega$ with $x\neq y$. After defining arbitrarily $f$ for $x=y$, we obtain $f\colon\Omega\times\Omega\times\R\times\R\to [0,+\infty)$ satisfying \eqref{fkgrowth}--\eqref{fkcontt}. The convergence in \eqref{weakstarfk} for every $s,t\in\R$ is obtained by approximation with $s,t\in\Q$.
\end{proof}

We define 
\begin{equation}\label{contconvcheck}
\widecheck F^\delta(u,A,B):= \int_{(A\times B)\setminus\Delta_\delta} f(x,y,u(x),u(y)) \dxy
\end{equation}
for $u\in L^q(\Omega)$ and $A,B\in\mathcal A$. Note that $u\mapsto \widecheck F^\delta(u,A,B)$ is continuous with respect to the strong convergence in $L^q(\Omega)$, since $f$ satisfies \eqref{fkgrowth}--\eqref{fkcontt}.

\begin{theorem}\label{contconvthm}
Assume that $f_k$ and $f$ satisfy \eqref{fkgrowth}--\eqref{fkcontt} and \eqref{weakstarfk}. Let $u_k,v_k, u, v\in L^q(\Omega)$ with $u_k\to u$ and $v_k\to v$ strongly in $L^q(\Omega)$. Then 
\begin{equation}\label{contconv}
\lim_{k\to+\infty}\int_{(A\times B)\setminus\Delta_\delta} f_k(x,y,u_k(x),v_k(y)) \dxy= \int_{(A\times B)\setminus\Delta_\delta} f(x,y,u(x),v(y)) \dxy
\end{equation}
for all $A,B\in\mathcal A$ and every $\delta>0$. In particular we have 
\begin{equation}\label{contconvcheck-1}
\lim_{k\to+\infty}\widecheck F^\delta_k(u_k,A,B)= \widecheck F^\delta(u,A,B).
\end{equation}
\end{theorem}

\begin{proof} Let $A,B\in\mathcal A$ be fixed, and let $\delta>0$.
 We first consider $\varphi,\psi$ simple functions, 
$$
\varphi=\sum_{i=1}^Ns_i {\bf 1}_{A_i} \quad\hbox{ and }\quad \psi=\sum_{j=1}^Mt_j {\bf 1}_{B_j},
$$
with $\{A_i\}$ and $\{B_j\}$ measurable partitions of $\Omega$ and $s_i,t_j\in\R$. 
We claim that 
\begin{equation}\label{int0}
\lim_{k\to+\infty}\int_{(A\times B)\setminus\Delta_\delta} f_k(x,y,\varphi(x),\psi(y)) \dxy= \int_{(A\times B)\setminus\Delta_\delta} f(x,y,\varphi(x),\psi(y)) \dxy.
\end{equation}

Indeed, setting $A'_i:=A_i\cap A$ and $B'_i:=B_i\cap B$ we have
$$
\int_{(A\times B)\setminus\Delta_\delta} f_k(x,y,\varphi(x),\psi(y)) \dxy=\sum_{i,j} \int_{(A'_i\times B'_j)\setminus\Delta_\delta} f_k(x,y,s_i,t_j) \dxy,
$$
and similarly for $f$. The claim follows, since $$\lim_{k\to+\infty}
 \int_{(A'_i\times B'_j)\setminus\Delta_\delta} f_k(x,y,s_i,t_j) \dxy= \int_{(A'_i\times B'_j)\setminus\Delta_\delta} f(x,y,s_i,t_j) \dxy
$$
by \eqref{weakstarfk}.

Now, in order to prove \eqref{contconv}, given $\e>0$ we fix two simple functions $\varphi,\psi$ such that
$$
\|u-\varphi\|_{L^q(\Omega)}\le \e \quad\hbox{ and }\quad \|v-\psi\|_{L^q(\Omega)}\le \e,
$$
and a constant $C$ such that
\begin{eqnarray*}
&\displaystyle c_\delta\int_{\Omega} (\alpha_\delta(y)+|v_k(y)|^q)\dy\le C, \qquad c_\delta\int_{\Omega} (\alpha_\delta(x)+|\varphi(x)|^q)\dx\le C,\\
& \displaystyle\Bigl(\int_\Omega \big(\beta_\delta(x)+(|\varphi(x)|\vee|u_k(x)|)^{q-1}\big)^{q\over q-1}\dx\Bigr)^{q-1\over q} \le C,\\
& \displaystyle\Bigl(\int_\Omega \big(\beta_\delta(y)+(|\psi(y)|\vee|v_k(y)|)^{q-1}\big)^{q\over q-1}\dy\Bigr)^{q-1\over q} \le C.
\end{eqnarray*}

By \eqref{fkconts} and  \eqref{fkcontt} we have 
\begin{eqnarray}\label{int1}\nonumber
&&\hskip-1cm\Bigl|\int_{(A\times B)\setminus\Delta_\delta} f_k(x,y,u_k(x),v_k(y)) \dxy- \int_{(A\times B)\setminus\Delta_\delta} f_k(x,y,\varphi(x),\psi(y)) \dxy\Bigr| \\ \nonumber
&&\hskip-1cm\le \int_{(A\times B)\setminus\Delta_\delta} \Bigl|f_k(x,y,u_k(x),v_k(y))- f_k(x,y,\varphi(x),v_k(y)) \Bigr|\dxy\\\nonumber
&& + \int_{(A\times B)\setminus\Delta_\delta} \Bigl|f_k(x,y,\varphi(x),v_k(y)) -f_k(x,y,\varphi(x),\psi(y))\Bigr| \dxy\\ \nonumber
&&\hskip-1cm\le c_\delta\int_{\Omega} (\alpha_\delta(y)+|v_k(y)|^q)\dy\int_\Omega \big(\beta_\delta(x)+(|\varphi(x)|\vee|u_k(x)|)^{q-1}\big)|u_k(x)-\varphi(x)|\dx\\ \nonumber
&&+ c_\delta\int_{\Omega} (\alpha_\delta(x)+|\varphi(x)|^q)\dx\int_\Omega \big(\beta_\delta(y)+(|\psi(y)|\vee|v_k(y)|)^{q-1}\big)|v_k(y)-\psi(y)|\dy\\
&&\hskip-1cm\le  C^2\big( \|u_k-\varphi\|_{L^q(\Omega)}+  \|v_k-\psi\|_{L^q(\Omega)}\big)\le C^2\big( \|u_k-u\|_{L^q(\Omega)}+  \|v_k-v\|_{L^q(\Omega)} +2\e\big)\,.
\end{eqnarray}
Similarly, we obtain
\begin{eqnarray}\label{int2}
\Bigl|\int_{(A\times B)\setminus\Delta_\delta} f(x,y,u(x),v(y)) \dxy- \int_{(A\times B)\setminus\Delta_\delta} f(x,y,\varphi(x),\psi(y)) \dxy\Bigr| \le C^22\e.
\end{eqnarray}
By \eqref{int0}--\eqref{int2} and the arbitrariness of $\e$, we obtain the claim \eqref{contconv}.
\end{proof}

\begin{remark}\label{rem-est-f}\rm Since
$$
0\le 
f_k(x,y,s,t)\le  {b_k(x,y)\over |x-y|^p} |z|^p+ c_\Delta(\alpha(x)+|s|^q)(\alpha(y)+|t|^q)
$$
by \eqref{doubleboundtf} and \eqref{growthbk}, the weak convergence in \eqref{weakconvbk} implies that
\begin{equation}\label{growth-f}
0\le 
f(x,y, s,t)\le  {b(x,y)\over |x-y|^p} |t-s|^p+ c_\Delta(\alpha(x)+|s|^q)(\alpha(y)+|t|^q)
\end{equation}
for almost all $x,y$ and all $s,t$.
\end{remark}

\begin{lemma} There exists $C>0$ such that if $A, B\in\AA$ and $0<\delta\le{\rm dist}(A\cup B, \mathbb R^d\setminus \Omega)$ then 
\begin{equation}\label{stima-incr}
\int_{\Delta_\delta\cap(A\times B)} {|u(x)-u(y)|^p\over|x-y|^p}\dxy\le C\delta^d\int_\Omega|\nabla u(x)|^p\dx
\end{equation}
for all $u\in W^{1,p}(\Omega)$. 
\end{lemma}

\begin{proof}
By approximation it is sufficient to deal with the case $u\in C^1(\Omega)$. Then for every $(x,y)\in\Delta_\delta\cap(A\times B)$, since the segment between $x$ and $y$ lies entirely in $\Omega$, we may write
$$
u(x)-u(y)=\int_0^1\nabla u(x+t(y-x))\cdot (y-x)dt,
$$ which gives
$$
{|u(x)-u(y)|^p\over|x-y|^p}\le \int_0^1|\nabla u(x+t(y-x))|^pdt.
$$
Hence,
\begin{eqnarray*}
&&\int_{\Delta_\delta\cap(A\times B)} {|u(x)-u(y)|^p\over|x-y|^p}\dxy\le
\int_0^1\int_{\Delta_\delta\cap(A\times B)} |\nabla u(x+t(y-x))|^p\dxy\,dt\\
&=&\int_0^{1\over 2}\int_B\int_{B_\delta(y)} |\nabla u(x+t(y-x))|^p\dx\dy\,dt+ \int_{1\over 2}^1\int_A\int_{B_\delta(x)} |\nabla u(x+t(y-x))|^p\dy\dx\,dt.
\end{eqnarray*}
By the change of variables $z=x+t(y-x)$ we can estimate the last two integrals  from above by
\begin{eqnarray*}
&&\int_0^{1\over 2}{1\over (1-t)^d}\int_B\int_{B_\delta(y)} |\nabla u(z)|^p\,dz\dy\,dt+ \int_{1\over 2}^1{1\over t^d}\int_A\int_{B_\delta(x)} |\nabla u(z)|^p\,dz\dx\,dt\\
&\le&2^{d-1}\int_B\int_{B_\delta(y)} |\nabla u(z)|^p\,dz\dy+2^{d-1}\int_A\int_{B_\delta(x)} |\nabla u(z)|^p\,dz\dx
\le \delta^d 2^d\omega_d\int_\Omega |\nabla u(z)|^p\,dz,
\end{eqnarray*}
which proves the claim.
\end{proof}

For all $A, B\in\mathcal A$ we set
\begin{equation}\label{effecek}
 F(u,A,B)=\int_{A\times B} f(x,y,u(x),u(y))\dxy
\end{equation}
for $u\in W^{1,p}(\Omega)$.

\begin{proposition}\label{contFcheck}
For every $u\in W^{1,p}(\Omega)$ and $A,B\in\AA$ with $A, B\subset\subset\Omega$  we have $F(u,A,B)<+\infty$ and the functional $u\mapsto F(u,A,B)$ is sequentially continuous with respect to the weak convergence in $W^{1,p}(\Omega)$.
\end{proposition}

\begin{proof} The finiteness of $ F(u,A,B)$ follows from the boundedness of $b$, \eqref{growth-f}, \eqref{stima-incr}, and the Sobolev embedding theorem.

Let $u_k$ be weakly converging to $u$ in $W^{1,p}(\Omega)$.
We claim that the energy on $(A\times B)\cap\Delta_\delta$ is small as $\delta\to0$, uniformly with respect to $k$. 
Indeed, by \eqref{growth-f} we have
\begin{eqnarray*}
&&\int_{(A\times B)\cap\Delta_\delta} f(x,y, u_k(x), u_k(y)\dxy\\ &\le & \|b\|_\infty \int_{\Delta_\delta}{|u_k(x)-u_k(y)|^p\over|x-y|^p}\dxy+ c_\Delta\int_{(A\times B)\cap\Delta_\delta}\hskip-1cm(\alpha(x)+|u_k(x)|^q)(\alpha(y)+|u_k(y)|^q)\dxy,
\end{eqnarray*}
so that, by \eqref{stima-incr}, 
\begin{eqnarray*}
&&\limsup_{k\to+\infty}\int_{(A\times B)\cap\Delta_\delta} f(x,y, u_k(x), u_k(y))\dxy\\ &\le & C\|b\|_\infty \delta^d \sup_k\int_\Omega|\nabla u_k|^p\dx+ c_\Delta\int_{(A\times B)\cap\Delta_\delta}(\alpha(x)+|u(x)|^q)(\alpha(y)+|u(y)|^q)\dxy,
\end{eqnarray*}
 which proves the claim since $(\alpha(x)+|u(x)|^q)(\alpha(y)+|u(y)|^q)\in L^1(A\times B)$.
 
 We can now conclude that $F(u_k,A,B)$ tends to $F(u,A,B)$. This follows  from the estimate above, since  the functional $\widecheck F^\delta(\cdot ,A,B)$ defined in \eqref{contconvcheck}  is strongly continuous in $L^q(\Omega)$ and $u_k$ converge strongly to $u$ in $L^q(\Omega)$. 
\end{proof}

\section{Behaviour on the diagonal}\label{diag}

Let $\DD$ be the countable dense subset of $\AA$ introduced at the end of Section \ref{localiz}. By the compactness of $\Gamma$-convergence there exists a subsequence of indices $k$, not relabelled, and a functional $\Phi^\delta\colon W^{1,p}(\Omega)\times \DD \times \DD \to[0,+\infty]$ such that for every $\delta>0$, $\delta\in\Q$,  and for every $A,B\in\DD$
we have 
\begin{equation}\label{gammaD-delta}
\Phi^\delta_k(\cdot,A,B)\hbox{ $\Gamma$-converges to $\Phi^\delta(\cdot,A,B)$}
\end{equation}
with respect to the weak topology of $W^{1,p}(\Omega)$.
 
We consider the functionals $ \Phi'_\delta,  \Phi''_\delta\colon W^{1,p}(\Omega)\times\AA\times\AA\to[0,+\infty]$  defined for all $A,B\in\AA$ by
 \begin{eqnarray}
 \Phi'_\delta(\cdot,A,B)=\Gamma\hbox{-}\liminf_{k\to+\infty} \Phi^\delta_k(\cdot,A,B),\\
  \Phi''_\delta(\cdot,A,B)=\Gamma\hbox{-}\limsup_{k\to+\infty} \Phi^\delta_k(\cdot,A,B),
  \end{eqnarray}
where the $\Gamma$-limits are taken with respect to the weak topology of $W^{1,p}(\Omega)$. By \eqref{gammaD-delta} we have $\Phi'_\delta(u,A,B)=\Phi''_\delta(u,A,B)$ for all $u\in W^{1,p}(\Omega)$ and $A,B\in\DD$.

For all $u\in W^{1,p}(\Omega)$ and $A\in\AA$ we define 
\begin{eqnarray}\label{Phideltap}
\Phi'_\Delta(u,A):= \inf_{\delta>0} \Phi'_\delta(u,A,A)= \lim_{\delta\to0^+} \Phi'_\delta(u,A,A),\\ \label{Phideltas}
\Phi''_\Delta(u,A):= \inf_{\delta>0} \Phi''_\delta(u,A,A)= \lim_{\delta\to0^+} \Phi''_\delta(u,A,A).
\end{eqnarray}
Note that by \eqref{gammaD-delta} we have $\Phi'_\Delta(u,A)=\Phi''_\Delta(u,A)$ for all $u\in W^{1,p}(\Omega)$ and $A\in\DD$.

\begin{proposition}\label{aa}
Let $\Phi'$ and $\Phi''$ be defined by \eqref{Phim} and  \eqref{Phip}, respectively. Then
\begin{equation}\label{46}
\Phi'(u,A,A)= \Phi'_\Delta(u,A)+ F(u, A,A),\quad
\Phi''(u,A,A)= \Phi''_\Delta(u,A)+ F(u, A,A)
\end {equation}
for every $u\in W^{1,p}(\Omega)$ and $A\in\mathcal A$.
\end{proposition}

\begin{proof} We have $$\Phi_k(u,A,A)=\Phi^\delta_k(u,A,A)+\widecheck F^\delta_k(u,A,A)$$
for every $u\in W^{1,p}(\Omega)$ and $A\in\mathcal A$.  From \eqref{contconvcheck-1}
we then deduce
$$\Phi'(u,A,A)=\Phi'_\delta(u,A,A)+\widecheck F^\delta(u,A,A),\qquad
\Phi''(u,A,A)=\Phi''_\delta(u,A,A)+\widecheck F^\delta(u,A,A).$$
Letting $\delta\to0$ we obtain \eqref{46}
by the definition of $\Phi'_\Delta$ and $\Phi''_\Delta$ and the integral form of $ F(u,A,A)$.
\end{proof}

\begin{proposition}\label{semicPhi}
For all $A\in\mathcal A$ with $A\subset\subset\Omega$ the functionals $\Phi'_\Delta(\cdot,A)$ and $\Phi''_\Delta(\cdot,A)$ are sequentially weakly lower semicontinuous in $W^{1,p}(\Omega)$.
\end{proposition}

\begin{proof} We note that $\Phi'(\cdot,A,A)$ and $\Phi''(\cdot,A,A)$ are weakly lower semicontinuous in $W^{1,p}(\Omega)$ as $\Gamma$-limits. Then the claim follows from Propositions \ref{contFcheck} and \ref{aa}.
\end{proof}

Note that $\Phi'_\Delta(u,A)\le \Phi'_\Delta(u,B)$ and $\Phi''_\Delta(u,A)\le \Phi''_\Delta(u,B)$ if $A\subset B$.
Since $\Phi'_\Delta(u,A)=\Phi''_\Delta(u,A)$ for all $u\in W^{1,p}(\Omega)$ and $A\in\DD$, we can define
$\Phi_\Delta\colon W^{1,p}(\Omega)\times\AA\to[0,+\infty]$ by
\begin{eqnarray}\label{PhiDelta}
\Phi_\Delta(u,A):= \sup\{\Phi'_\Delta(u,B): B\in\AA, B\subset\subset A\}= \sup\{\Phi''_\Delta(u,B): B\in\AA, B\subset\subset A\}.
\end{eqnarray}
Note that $A\mapsto \Phi_\Delta(u,A)$ is inner regular; i.e.,
\begin{eqnarray}\label{inner-reg}
\Phi_\Delta(u,A)= \sup\{\Phi_\Delta(u,B): B\in\AA, B\subset\subset A\}.
\end{eqnarray}
Moreover, by Proposition \ref{semicPhi} 
\begin{eqnarray}\label{inner-reg-lsc}
\Phi_\Delta(\cdot,A)\hbox{ is sequentially weakly lower semicontinuous in $W^{1,p}(\Omega)$}
\end{eqnarray}
for every $A\in\mathcal A$.

\begin{proposition}\label{estab} There exist $c_2>0$ such that for all $u\in W^{1,p}(\Omega)$ we have
\begin{eqnarray}\label{estdalla}
\Phi_\Delta(u,A)\le \Phi''_\Delta(u,A)\le\int_A(a(x)+c_2|\nabla u(x)|^p)\dx
\end{eqnarray}
for all $A\in\AA$.
\end{proposition}

\begin{proof}Let $u\in C^1(\overline\Omega)$, let $A\in\AA$ with $A\subset \subset \Omega$ and $\mathcal L^d(\partial A)=0$, and $\delta>0$ with $\delta\le {\rm dist} (A,\mathbb R^d\setminus\Omega)$. Then by Taylor's formula there exists a modulus of continuity $\omega$ such that
\begin{eqnarray}\label{stimac}\nonumber
\Phi^\delta_k(u,A,A)&\le& \int_A(c_1|\nabla u (x)|^p+ a(x))\dx+\int_{(A\times A)\cap \Delta_\delta} {b_k(x,y)\over |x-y|^p} |u(x)-u(y)|^p\dxy\\ \nonumber
&& + c_\Delta\int_{(A\times A)\cap \Delta_\delta} (\alpha(x)+|u(x)|^q)(\alpha(y)+|u(y)|^q)\dxy
\\ \nonumber
&\le& \int_A(c_1|\nabla u (x)|^p+ a(x))\dx+\int_{(A\times A)\cap \Delta_\delta} b_k(x,y)|\nabla u(x)+\omega(\delta)|^p\dxy\\ \nonumber
&& + c_\Delta\int_{(A\times A)\cap \Delta_\delta} (\alpha(x)+|u(x)|^q)(\alpha(y)+|u(y)|^q)\dxy
\\ \nonumber
&\le& \int_A(c_1|\nabla u (x)|^p+ a(x))\dx+2^{p-1}\int_{(A\times A)\cap \Delta_\delta} b_k(x,y)|\nabla u(x)|^p\dxy \\
&& + c_\Delta\int_{(A\times A)\cap \Delta_\delta} (\alpha(x)+|u(x)|^q)(\alpha(y)+|u(y)|^q)\dxy\nonumber\\ && +2^{p-1} c_b\mathcal L^d(A) \omega^p(\delta),
\end{eqnarray}
where in the last inequality we have used the boundedness assumption \eqref{cibi}.

Letting $k\to+\infty$ we obtain
\begin{eqnarray*}
\Phi''_\delta(u,A)
&\le& \int_A(c_1|\nabla u (x)|^p+ a(x))\dx+2^{p-1}\int_{(A\times A)\cap \Delta} b_\Delta(x,y)|\nabla u(x)|^p\,d\HH^d(x,y)\\
&&+2^{p-1}\int_{(A\times A)\cap \Delta_\delta} b(x,y)|\nabla u(x)|^p\dxy \\
&& + c_\Delta\int_{(A\times A)\cap \Delta_\delta} (\alpha(x)+|u(x)|^q)(\alpha(y)+|u(y)|^q)\dxy+2^{p-1}C \omega^p(\delta),
\end{eqnarray*}
while, letting $\delta\to0$, we get
\begin{eqnarray*}
\Phi''_\Delta(u,A)
\le \int_A(c_1|\nabla u (x)|^p+ a(x))\dx+2^{p-1}\int_{(A\times A)\cap \Delta} b_\Delta(x,y)|\nabla u(x)|^p\,d\HH^d(x,y),
\end{eqnarray*}
which proves the claim e.g.~with $c_2=c_1+d 2^{p-1}\|b_\Delta\|_\infty$ under our hypotheses on $u$ and $A$.

For a general $u\in W^{1,p}(\Omega)$ we fix $v\in W^{1,p}_0(\Omega)$ such that $u=v$ almost everywhere in $A$, which is possible since we assume that $A\subset \subset \Omega$. Approximating $v$ with functions in $C^\infty_c(\Omega)$, we conclude the proof by the lower semicontinuity of $\Phi''_\Delta(\cdot,A)$. To remove the hypothesis $A\subset\subset \Omega$ we use the inner regularity given by \eqref{inner-reg}.
\end{proof}

The following proposition will be used to relate the values $\Phi'(u,A,B)$ and $\Phi''(u,A,B)$  to $ F(u, A,B)+\Phi'_\Delta(u,A\cap B)$ and $ F(u, A,B)+\Phi''_\Delta(u,A\cap B)$, respectively.
We use the notation $A_\eta:=\{x\in\Omega: {\rm dist}(x,A)<\eta\}$.

\begin{proposition} 
Let $\Phi'$ and $\Phi''$ be defined by \eqref{Phim} and  \eqref{Phip}, respectively, let 
$\Phi'_\Delta$ and $\Phi''_\Delta$ be defined by
\eqref{Phideltap} and \eqref{Phideltas}, respectively, let $F$ be defined in \eqref{effecek}, let $A,B\in\mathcal A$, let $u\in W^{1,p}(\Omega)$, and let  $\eta>0$. Then we have
\begin{eqnarray}\label{phid-1}
 F(u, A,B)+\Phi'_\Delta(u,A\cap B)
\le
\Phi'(u,A,B)
\le  F(u, A,B)+\Phi'_\Delta(u,A_\eta\cap B_\eta),
\\ \label{phid-2}
 F(u, A,B)+\Phi''_\Delta(u,A\cap B)
\le
\Phi''(u,A,B)
\le  F(u, A,B)+\Phi''_\Delta(u,A_\eta\cap B_\eta).
\end{eqnarray}
\end{proposition}

\begin{proof}
Let $0<\delta\le\eta$. Note that $$(A\times B)\cap\Delta_\delta\subset \bigl((A_\delta\cap B_\delta)\times (A_\delta\cap B_\delta)\bigr)\cap\Delta_\delta\subset \bigl((A_\eta\cap B_\eta)\times (A_\eta\cap B_\eta)\bigr)\cap\Delta_\delta.$$ Let $u_k$ be a sequence in $W^{1,p}(\Omega)$ weakly converging to $u$. Then
\begin{eqnarray*}
\Phi_k(u_k,A,B)&=& \widecheck F_k^\delta(u_k, A,B)+\Phi_k^\delta(u_k,A,B)\\
&\le& \widecheck F_k^\delta(u_k, A,B)+\Phi_k^\delta(u_k,A_\eta\cap B_\eta,A_\eta\cap B_\eta).
\end{eqnarray*}
By \eqref{contconvcheck-1} we then obtain
\begin{eqnarray*}
\liminf_{k\to+\infty}\Phi_k(u_k,A,B)
\le \widecheck F^\delta(u, A,B)+\liminf_{k\to+\infty}\Phi_k^\delta(u_k,A_\eta\cap B_\eta,A_\eta\cap B_\eta),
\end{eqnarray*}
and, by the arbitrariness of $u_k$,
\begin{eqnarray*}
\Phi'(u,A,B)
\le \widecheck F^\delta(u, A,B)+\Phi'_\delta(u,A_\eta\cap B_\eta,A_\eta\cap B_\eta).
\end{eqnarray*}
Taking the limit as $\delta\to0$ we then get the second inequality in \eqref{phid-1}.
As for the first inequality in \eqref{phid-1}, we observe that
\begin{eqnarray*}
\Phi_k(u_k,A,B)&=& \widecheck F_k^\delta(u_k, A,B)+\Phi_k^\delta(u_k,A,B)\\
&\ge& \widecheck F_k^\delta(u_k, A,B)+\Phi_k^\delta(u_k,A\cap B,A\cap B),
\end{eqnarray*}
and proceed as above. Similarly, we prove \eqref{phid-2}.
\end{proof}

\section{Integral representation of the term on the diagonal}\label{intrepsec}
In this section we prove that the functional $\Phi_\Delta$ introduced in \eqref{PhiDelta} can be represented as an integral; more precisely, we shall prove the following theorem.

\begin{theorem}\label{mainth} Let $f_k$ satisfy \eqref{fkgrowth}--\eqref{invariance}, let $g_k$ satisfy \eqref{usual}, and let $\Phi_\Delta$ be defined by \eqref{PhiDelta}. Then there exists a Borel function $g:\Omega\times\R^d\to[0,+\infty)$, convex in the second variable, such that
\begin{equation}
\Phi_\Delta(u,A)=\int_A g(x,\nabla u(x))\dx
\end{equation}
for all $u\in W^{1,p}(\Omega)$ and $A\in \AA$. 
\end{theorem}

The proof of the theorem will be achieved after some technical results as follows.
We start with proving a subadditivity property.
\begin{proposition}\label{subaddit}
Let $u\in W^{1,p}(\Omega)$. Let $A,B\in\AA$. Then we have
\begin{equation}\label{subad=Phi}
\Phi_\Delta(u, A\cup B)\le \Phi_\Delta(u, A)+\Phi_\Delta(u, B).
\end{equation}
\end{proposition}

By \eqref{PhiDelta} the claim of Proposition \ref{subaddit} is a consequence of the following lemma.

\begin{lemma}\label{lemmasub}
Let $u\in W^{1,p}(\Omega)$ and let $A,A',B\in\AA$ with $A'\subset\subset A$. Then we have
\begin{equation}\label{subad}
\Phi''_\Delta(u, A'\cup B)\le \Phi''_\Delta(u, A)+\Phi''_\Delta(u, B).
\end{equation}
\end{lemma}

Lemma \ref{lemmasub} will be a consequence of the following general result, which will be used to join recovery sequences.

\begin{lemma}\label{lemmasub-0}
Let $u\in W^{1,p}(\Omega)$ and let $v_k$ and $w_k$ be sequences weakly converging to $u$ in $W^{1,p}(\Omega)$. Let $A,A',A'',B\in\AA$ with $A'\subset\subset A''\subset\subset A$. Then for all $\sigma>0$ there exist $\delta_\sigma>0$  and a sequence  $u_k$ converging to $u$ in $W^{1,p}(\Omega)$ such that 
$u_k=v_k$ on $A'$, $u_k=w_k$ on $\Omega\setminus A''$, and 
\begin{equation}\label{subad-0}
\limsup_{k\to+\infty}\Phi_k^\delta(u_k, A'\cup B, A'\cup B)\le (1+\sigma)\limsup_{k\to+\infty}\bigl(\Phi_k^\delta(v_k, A,A)+\Phi_k^\delta(w_k, B,B)\bigr) + \sigma
\end{equation}
 for every $0<\delta\le\delta_\sigma$.
\end{lemma}

\begin{proof}
Given $m\in \mathbb N$, let $A_0,\ldots, A_m\in\AA$ with
$A'=A_0\subset\subset A_1\subset\subset A_2\subset\subset \cdots \subset\subset A_m=A''$, and let $\varphi_i $ be a cut-off function between $A_{i-1}$ and $A_i$. We define
$u^i_k=\varphi_i v_k+(1-\varphi_{i})w_k$. 

We now estimate $F^\delta_k(u^i_k, A'\cup B, A'\cup B)$. Note that for all $i\in\{1,\ldots, m\}$, after setting
$$
D^i_1:= (A'\cup B)\cap A_{i-1},\qquad D^i_2:= B\cap (A_i\setminus A_{i-1}),\qquad 
D^i_3:= B \setminus A_i,
$$
we can write
$A'\cup B= D^i_1\cup D^i_2\cup D^i_3$;
hence, we can split
$$
(A'\cup B)\times (A'\cup B)=\bigcup_{j=1}^3\bigcup_{\ell=1}^3 (D^i_j\times D^i_\ell)
$$
and separately examine the integrals on each of the sets on the right-hand side.

We have
\begin{eqnarray}\label{D11}\nonumber
\int_{(D^i_1\times D^i_1)\cap\Delta_\delta} \hskip-1.2cm f_k(x,y,u^i_k(x), u^i_k(y))\dxy
&\le&\int_{(A_{i-1}\times A_{i-1})\cap\Delta_\delta}\hskip-1.8cm f_k(x,y,v_k(x), v_k(y))\dxy\\
&\le&\int_{(A\times A)\cap\Delta_\delta} \hskip-1cm f_k(x,y,v_k(x), v_k(y))\dxy.
\end{eqnarray}
In the same way
\begin{eqnarray}\label{D33}\nonumber
\int_{(D^i_3\times D^i_3)\cap\Delta_\delta} \hskip-1.2cmf_k(x,y,u^i_k(x), u^i_k(y))\dxy
&=&\int_{((B\setminus A_i)\times(B\setminus A_i))\cap\Delta_\delta} \hskip-2.2cmf_k(x,y,w_k(x), w_k(y))\dxy\\
&\le&\int_{(B\times B)\cap\Delta_\delta}  \hskip-1cm f_k(x,y,w_k(x), w_k(y))\dxy.
\end{eqnarray}

In order to estimate the other pairs $(j,\ell)$, it is convenient to note that
\begin{eqnarray*} 
&\hskip-3cm u^i_k(x)- u^i_k(y)=\varphi_i(x) (v_k(x)-v_k(y))
+(1-\varphi_{i}(x))(w_k(x)-w_k(y))\\
&\hskip2cm 
+\varphi_i(x)v_k(y)
+(1-\varphi_i(x))w_k(y)-\varphi_i(y) v_k(y)-(1-\varphi_{i}(y))w_k(y)
\\
&=\varphi_i(x) (v_k(x)-v_k(y))
+(1-\varphi_{i}(x))(w_k(x)-w_k(y))
+(\varphi_i(x)-\varphi_i(y)) (v_k(y)-w_k(y)).
\end{eqnarray*} 
Note that by the convexity of $\widetilde f_k$ we have 
\begin{eqnarray}\label{Djell-co}
\nonumber
\widetilde f_k(x,y,u^i_k(x)- u^i_k(y))
&\le&{1\over 2}\widetilde f_k(x,y,2(v_k(x)-v_k(y)))+{1\over 2}\widetilde f_k(x,y,2(w_k(x)-w_k(y)))\\
&&+{1\over 2}\widetilde f_k(x,y, 2
(\varphi_i(x)-\varphi_i(y)) (v_k(y)-w_k(y))).
\end{eqnarray}
In view of \eqref{doubleboundtf} we also define 
\begin{eqnarray}\label{zetak}
&&z_k(x):= (\alpha(x)+|v_k(x)|^q+|w_k(x)|^q),
\\
\label{r12}
&&r^{ik}_{j\ell}(\delta):= c_\Delta\int_{(D^i_j\times D^i_\ell)\cap\Delta_\delta} z_k(x) z_k(y)\dxy.
\end{eqnarray}
Then, using \eqref{doubleboundtf}, \eqref{growthbk}, and \eqref{Djell-co}, we get
%
\begin{eqnarray}\label{Djell}
\nonumber
&&\int_{(D^i_j\times D^i_\ell)\cap\Delta_\delta}f_k(x,y,u^i_k(x), u^i_k(y))\dxy\\ \nonumber
\\ \nonumber
&\le&{1\over 2}\int_{(D^i_j\times D^i_\ell)\cap\Delta_\delta}\widetilde f_k(x,y,2(v_k(x)-v_k(y)))\dxy\\
\nonumber
&&+{1\over 2}\int_{(D^i_j\times D^i_\ell)\cap\Delta_\delta}\widetilde f_k(x,y,2(w_k(x)-w_k(y)))\dxy\\ \nonumber
&&+{1\over 2}\int_{(D^i_j\times D^i_\ell)\cap\Delta_\delta}\widetilde f_k(x,y, 2
(\varphi_i(x)-\varphi_i(y)) (v_k(y)-w_k(y))))\dxy
+r^{ik}_{j\ell}(\delta)
\\ \nonumber
&\le&{1\over 2} c_\Delta \int_{(D^i_j\times D^i_\ell)\cap\Delta_\delta}f_k(x,y,v_k(x),v_k(y))\dxy\\
\nonumber
&&+{1\over 2} c_\Delta \int_{(D^i_j\times D^i_\ell)\cap\Delta_\delta}f_k(x,y,w_k(x),w_k(y))\dxy\\
&&+C_m \int_{(D^i_j\times D^i_\ell)\cap\Delta_\delta}b_k(x,y) 
|v_k(y)-w_k(y)|^p\dxy
+r^{ik}_{j\ell}(\delta),
\end{eqnarray}
where  $L_m$ is a common Lipschitz constant for all $\varphi_i$ and $C_m:=2^{p-1}L_m^p$. 

 We now remark that there exists $\delta_0=\delta_0(A_1\ldots,A_m)>0$ such that for $0<\delta\le\delta_0$ 
{\color{blue}}
for all $i$ we have
\begin{equation}\label{dis-empt}
(D^i_1\times D^i_3)\cap \Delta_\delta=(D^i_3\times D^i_1)\cap \Delta_\delta=\emptyset,
\end{equation}
and for all $\ell\in\{1,2,3\}$ and $i\in\{2,\ldots, m-1\}$ 
$$
(D^i_2\times D^i_\ell)\cap \Delta_\delta\subset( D^i_2\times D^{}_2) \cap \Delta_\delta.
$$
where 
$ D_2= (A''\setminus A')\cap B$.

We then deduce that
\begin{eqnarray*}
&&\sum_{i=2}^{m-1}\sum_{\ell=1}^3 \int_{(D^i_2\times D^i_\ell)\cap\Delta_\delta}f_k(x,y,u^i_k(x), u^i_k(y))\dxy\\
&\le& \sum_{i=1}^{m}\biggl(
 {3\over 2} c_\Delta \int_{(D^i_2\times D^{}_2) \cap \Delta_\delta}(f_k(x,y,v_k(x),v_k(y))+f_k(x,y,w_k(x),w_k(y))
 \dxy\\
&&+3 C_m \int_{(D^i_2\times D^{}_2) \cap \Delta_\delta}b_k(x,y) 
|v_k(y)-w_k(y)|^p\dxy+\int_{(D^i_2\times D^{}_2)\cap\Delta_\delta} z_k(x)z_k(y)\dxy\biggr)\\
&=& 
 {3\over 2} c_\Delta \int_{(D^{}_2\times D^{}_2) \cap \Delta_\delta}(f_k(x,y,v_k(x),v_k(y))+f_k(x,y,w_k(x),w_k(y)))\dxy\\
&&+3C_m \int_{(D^{}_2\times D^{}_2)  \cap \Delta_\delta}b_k(x,y) 
|v_k(y)-w_k(y)|^p\dxy+\int_{(D^{}_2\times D^{}_2)\cap\Delta_\delta} z_k(z)z_k(y)\dxy.
\end{eqnarray*}
Similarly, for all $j\in\{1,2,3\}$ and $i\in\{2,\ldots, m-1\}$ 
$$
(D^i_j\times D^i_2)\cap \Delta_\delta
\subset \bigl(D^{}_2\times D^i_2\bigr)  \cap \Delta_\delta,
$$
and the same estimate holds for 
$$\sum_{i=2}^{m-1}\sum_{j=1}^3 \int_{(D^i_j\times D^i_{2})\cap\Delta_\delta}f_k(x,y,u^i_k(x), u^i_k(y))\dx.$$

On the other hand, a well-known argument (see \cite[Theorem 19.1]{DM}) gives the existence of a constant $\widehat c$ independent of $m$, and a constant $\widehat c_m$, depending on $m$, such that
\begin{eqnarray*}
\sum_{i=1}^m \int_{D^i_2} g_k(x,\nabla u^i_k(x))\dx&\le &
\widehat c\int_{D^{}_2} (g_k(x, \nabla v_k(x))+ g_k(x, \nabla w_k(x)))\dx\\
&&+
\widehat c\int_{D^{}_2}a(x)\dx+\widehat c_m \int_{D^{}_2}|v_k(x)-w_k(x)|^p\dx.
\end{eqnarray*}

We deduce that there exists $i=i_k\in\{2,\ldots, m-1\}$ such that
\begin{eqnarray*}\label{stimaD2}\nonumber
&&\sum_{\ell=1}^3 \int_{(D^i_2\times D^i_\ell)\cap\Delta_\delta}f_k(x,y,u^i_k(x), u^i_k(y))\dxy\\ \nonumber
&&+ \sum_{j=1}^3 \int_{(D^i_j\times D^i_{2})\cap\Delta_\delta}f_k(x,y,u^i_k(x), u^i_k(y))\dxy+ \int_{D^i_2} g_k(x,\nabla u^i_k(x))\dx \end{eqnarray*}
\begin{eqnarray}\label{stimaD2}\nonumber
&\le &{2\over m-2} \biggl( {3\over 2} c_\Delta \int_{(A\times A) \cap \Delta_\delta}f_k(x,y,v_k(x),v_k(y))\dxy\\ \nonumber
&&+{3\over 2} c_\Delta \int_{(B\times B)  \cap \Delta_\delta}f_k(x,y,w_k(x),w_k(y))\dxy+\int_{(D^{}_2\times D^{}_2) \cap\Delta_\delta} z_k(x)z_k(y)\dxy\\ \nonumber
&&+\widehat c\int_{D^{}_2} (g_k(x, \nabla v_k(x))+ g_k(x, \nabla w_k(x)))\dx\\
&&+
\widehat c\int_{D^{}_2}a(x)\dx+\Bigl(3 C_m c_b+\widehat c_m\Bigr) \int_{A''\cap B}|v_k(x)-w_k(x)|^p\dx\biggr),
\end{eqnarray}
where $c_b$ is the constant in \eqref{cibi}.

We fix $m\in\mathbb N$ such that
\begin{equation}\label{aaa1}
\frac{3 c_\Delta+2\widehat c+2C_u}{m-2}\le \sigma,
\end{equation}
where 
$$ C_u:=\Bigl(\int_{D^{}_2}(\alpha(x)+2|u(x)|^q)\dx\Bigr)^2+\widehat c\int_{D^{}_2} a(x)\dx,
$$
and define $u_k:=u^{i_k}_k$ and $\delta_\sigma=\delta_0(A_1\ldots A_m)$ such that  \eqref{dis-empt} holds. 

Then, by \eqref{D11}, \eqref{D33} and the analogous estimates for the terms involving $g_k$, if $\delta\le \delta_\sigma$ then  \eqref{dis-empt} gives
\begin{eqnarray}\label{stimafin}\nonumber
\Phi^\delta_k(u_k,A'\cup B,A'\cup B)&\le&\Phi^\delta_k(v_k,A,A)+\Phi^\delta_k(w_k,B,B)\\ \nonumber
&&\hskip-3cm +\sum_{\ell=1}^3 \int_{(D^i_2\times D^i_\ell)\cap\Delta_\delta}\hskip-1cm f_k(x,y,u^i_k(x), u^i_k(y))\dxy\\ && \hskip-3cm + \sum_{j=1}^3 \int_{(D^i_j\times D^i_{2})\cap\Delta_\delta}\hskip-1cm f_k(x,y,u^i_k(x), u^i_k(y))\dxy+\int_{D^i_2} g_k(x,\nabla u^i_k(x))\dx.
\end{eqnarray}

Since $v_k$ and $w_k$ tend to $u$ strongly in $L^p(D^{}_2)$ and in $L^q(D^{}_2)$ by Rellich's theorem, we have 
$$
\lim_{k\to+\infty} \int_{D^{}_2}|v_k(y)-w_k(y)|^p\dy=0
$$
and, by \eqref{zetak}, 
\begin{eqnarray*}
&&\lim_{k\to+\infty}
\int_{(D^{}_2 \times D^{}_2) \cap\Delta_\delta}  z_k(x)z_k(y)\dxy+
\widehat c\int_{D^{}_2}a(x)\dx\le C_u.
\end{eqnarray*}

Passing to the limit as  $k\to+\infty$ in \eqref{stimafin} and using \eqref{stimaD2}
we obtain \eqref{subad-0}. Note that this inequality is proved only for $\delta\le\delta_\sigma$ such that  \eqref{dis-empt} holds.
\end{proof}

\begin{proof}[Proof of Lemma \rm\ref{lemmasub}] It suffices to prove that for all $\sigma>0$ there exists 
$\delta_\sigma>0$ such that 
\begin{equation}\label{subad-d}
\Phi''_\delta(u, A'\cup B, A'\cup B)\le (1+\sigma)(\Phi''_\delta(u, A,A)+\Phi''_\delta(u, B,B)) + \sigma
\end{equation}
 for every $0<\delta\le\delta_\sigma$. To that end it suffices to apply Lemma \ref{lemmasub-0} choosing recovery sequences of $\Phi''_\delta(u, A,A)$ and $\Phi''_\delta(u, B,B)$  as $v_k$ and $w_k$, respectively. 
\end{proof}

\begin{corollary}\label{coro}
Let $u\in W^{1,p}(\Omega)$ and $A\in\AA$. Then 
\begin{equation}
\Phi'_\Delta(u, A)=\Phi''_\Delta(u, A)=\Phi_\Delta(u, A).
\end{equation}
\end{corollary}

\begin{proof} Since by definition we have
$
\Phi_\Delta(u, A)\le\Phi'_\Delta(u, A)\le\Phi''_\Delta(u, A)$,
we only have to prove that 
\begin{equation}\label{claim6}
\Phi''_\Delta(u, A)\le\Phi_\Delta(u, A).
\end{equation}
With fixed $\e>0$ let $K$ be a compact subset of $A$ such that
\begin{equation}\label{estest}
\int_{A\setminus K}(a(x)+c_2|\nabla u(x)|^p)\dx<\e
\end{equation}
with $c_2$ as in Proposition \ref{estab}.

We fix $A', A''\in\AA$ with $K\subset A'\subset\subset A''\subset\subset A$. By Lemma \ref{lemmasub}, we have
$$
\Phi''_\Delta(u, A)\le \Phi''_\Delta(u, A'')+\Phi''_\Delta(u, A\setminus K)\le \Phi_\Delta(u, A)+\e
$$
by \eqref{estdalla} and \eqref{estest}. By the arbitrariness of $\e$ we obtain claim \eqref{claim6}.
\end{proof}

We next prove a superadditivity property.
\begin{proposition}\label{superadd}
Let $u\in W^{1,p}(\Omega)$. Let $A,B\in\AA$ with $A\cap B=\emptyset$. Then we have
\begin{equation}\label{superadd-Phi}
\Phi_\Delta(u, A\cup B)\ge \Phi_\Delta(u, A)+\Phi_\Delta(u, B).
\end{equation}
\end{proposition}

\begin{proof} We fix $A',B'\in\AA$ with $A'\subset\subset A$ and  $B'\subset\subset B$, and let 
$\delta_0:={\rm dist}(A',B')$.
We claim that 
\begin{equation}\label{claim-1}
\Phi'_\delta(u,A', A')+ \Phi'_\delta(u,B',B')\le \Phi'_\delta(u,A'\cup B',A'\cup B')
\end{equation}
for all $0<\delta\le \delta_0$.

Note that for such $\delta$ we have 
\begin{equation}\label{subins}
\bigl((A'\times A')\cap\Delta_\delta\bigr)\cup \bigl((B'\times B')\cap\Delta_\delta\bigr) = \bigl((A'\cup B')\times (A'\cup B')\bigr)\cap\Delta_\delta.
\end{equation}
If $u_k\to u$ is a sequence such that 
$$
\Phi'_\delta(u,A'\cup B',A'\cup B')=\liminf_{k\to+\infty} \Phi_k^\delta(u_k,A'\cup B',A'\cup B'),
$$
by \eqref{subins} we have 
$$
 \Phi_k^\delta(u_k,A'\cup B',A'\cup B')= \Phi_k^\delta(u_k,A',A')+ \Phi_k^\delta(u_k,B',B'),
 $$
 so that \eqref{claim-1} is proved by the definition of $\Gamma$-liminf.
 
 Taking the limit as $\delta\to0$ in \eqref{claim-1} we obtain 
\begin{equation}\label{claim-2}
\Phi'_\Delta(u,A')+ \Phi'_\Delta(u,B')\le \Phi'_\Delta(u,A'\cup B').
\end{equation}
The claim of the proposition eventually follows by the definition of $\Phi_\Delta$ in \eqref{PhiDelta}.
\end{proof}

\begin{proposition}\label{inv-co}
Let $u\in W^{1,p}(\Omega)$, let $A\in\AA$, and let $r\in \R$. Then we have
\begin{equation}\label{invar=Phi}
\Phi_\Delta(u+r, A)= \Phi_\Delta(u, A).
\end{equation}
\end{proposition}

\begin{proof} It is enough to prove that \begin{equation}\label{invar=Phi-}
\Phi_\Delta(u+r, A)\le \Phi_\Delta(u, A).
\end{equation}
With fixed $A'\in \AA$, with $A'\subset\subset A$, and $\delta>0$, let $u_k\rightharpoonup u$ weakly in $W^{1,p}(\Omega)$ such that 
\begin{equation}\label{rec=Phi-}
\Phi''_\delta(u, A',A')= \limsup_{k\to+\infty} \Phi_k^\delta(u_k, A',A').
\end{equation}
By the definition of $\Gamma$-limsup we have 
\begin{equation}\label{gammalimsup=Phi-}
\Phi''_\delta(u+r, A',A')\le \limsup_{k\to+\infty} \Phi_k^\delta(u_k+r, A',A').
\end{equation}
By \eqref{ficappadelta} we have 
\begin{eqnarray}\label{dise-no}\nonumber
\Phi_k^\delta(u_k+r, A',A')&=&F_k^\delta(u_k+r, A',A')+ G_k(u_k+r,A')\\ \nonumber
&
\le& F_k^\delta(u_k, A',A')+ G_k(u_k,A') +r_k^\delta\\
&
=& \Phi_k^\delta(u_k, A',A')+r_k^\delta,
\end{eqnarray}
where,  using \eqref{invariance},
\begin{eqnarray*}
r_k^\delta&:=&\int_{\Delta_\delta}\alpha_r(x,y)\dxy+ \omega_r(\delta) F_k^\delta(u_k, A',A')\\
&&+ c_\Delta\int_{(A'\times A')\cap \Delta_\delta} (\alpha(x)+|u_k(x)|^q)(\alpha(y)+|u_k(y)|^q)\dxy.
\end{eqnarray*}
Since $A'\subset\subset A$ we have $u_k\to u$ in $L^q(A')$ by Rellich's theorem. 
By \eqref{gammalimsup=Phi-} and \eqref{dise-no} we have
\begin{eqnarray*}\label{gammalimsup-Phi-}
\Phi''_\delta(u+r, A',A')&\le& (1+\omega_r(\delta) )\Phi''_\delta(u, A',A')+\int_{\Delta_\delta}\alpha_r(x,y)\dxy \\
&&+ c_\Delta\int_{(A'\times A')\cap \Delta_\delta} (\alpha(x)+|u(x)|^q)(\alpha(y)+|u(y)|^q)\dxy.
\end{eqnarray*}
Letting $\delta\to 0$ we have 
$$
\Phi''_\Delta(u+r, A')\le \Phi''_\Delta(u,A').
$$
Taking the supremum as $A'\subset\subset A$ we obtain \eqref{invar=Phi-} and conclude the proof.
\end{proof}

The next proposition states a locality property for $\Phi_\Delta$.
\begin{proposition}\label{locality}
Let $A\in\AA$ and $u,v\in W^{1,p}(\Omega)$ with $u=v$ almost everywhere in $A$. Then
$\Phi_\Delta(u,A)=\Phi_\Delta(v,A)$.
\end{proposition}

\begin{proof} It is sufficient to prove that 
$\Phi_\Delta(u,A)\le\Phi_\Delta(v,A)$.
By the definition of \eqref{PhiDelta}, we need to prove that for every $A'\in\AA$ with $A'\subset\subset A$ we have 
\begin{equation}\label{loc0d}
\Phi''_\delta(u,A',A')\le\Phi''_\delta(v,A',A')
\end{equation}
for $\delta>0$ small enough. With fixed $A'$ and $A''$ with $A'\subset\subset A''\subset\subset A$ we set 
$\delta_0={\rm dist}(A', \R^d\setminus A'')$. Let $0<\delta<\delta_0$ and let $v_k\rightharpoonup v$ weakly  in $W^{1,p}(\Omega)$ be such that
$$
\Phi''_\delta(v,A',A')
=\limsup_{k\to+\infty} \Phi_k^\delta(v_k, A',A').
$$
Let $\varphi$ be a cut-off function between $A''$ and $A$ and define $u_k=\varphi v_k+(1-\varphi)u$. Since $u=v$ almost everywhere in $A$ we have that $u_k\rightharpoonup u$ weakly in $W^{1,p}(\Omega)$. Using this test sequence we obtain
$$
\Phi''_\delta(u,A',A')
\le\limsup_{k\to+\infty} \Phi_k^\delta(u_k, A',A')= \limsup_{k\to+\infty} \Phi_k^\delta(v_k, A',A')=\Phi''_\delta(v,A',A'),
$$
where the first equality holds since $u_k=v_k$ in $A''$ and $\delta<\delta_0$.
\end{proof}

\begin{proof}[Proof of Theorem {\rm \ref{mainth}}] 
The functional $\Phi_\Delta:W^{1,p}(\Omega)\times\AA\to [0,+\infty)$ is local (Proposition \ref{locality}).
For every $u\in W^{1,p}(\Omega)$ the set function $A\mapsto\Phi_\Delta(u,A)$ is 
increasing, inner regular  (by \eqref{inner-reg}), subadditive (by Proposition \ref{subaddit}), and 
superadditive (by Proposition \ref{superadd}).
Therefore, by the De Giorgi-Letta Measure Criterion it is the restriction to $\AA$ of a Borel measure.

For every $A\in\AA$ the function $u \mapsto\Phi_\Delta(u,A)$ is lower semicontinuous with respect to the weak convergence in 
$W^{1,p}(\Omega)$ (by \eqref{inner-reg-lsc}), and satisfies a $p$-growth condition from above (by Proposition \ref{estab}).
Moreover, we have the invariance by addition of constants
$\Phi_\Delta(u+r,A)=\Phi_\Delta(u,A)$ (by Proposition \ref{inv-co}).

We finally note that \eqref{usual} implies that  
$\Phi_\Delta(u,A)\ge c_0\int_A|\nabla u|^p\dx$, which gives the lower semicontinuity of $\Phi_\Delta(\cdot,A)$ also with respect to the strong $L^p(\Omega)$ convergence. We can then apply the integral-representation result Theorem 20.1 in \cite{DM}
and obtain the claim of the theorem from the properties above.
\end{proof}

\section{Global representation}
The following result shows in particular that the $\Gamma$-limit of $\Phi_k$ is equal to
$$
\int_{\Omega\times\Omega} f(x,y,u(x),u(y))\dxy +\int_\Omega g(x,\nabla u(x))\dx.
$$

\begin{theorem}\label{thglo} Let $u\in W^{1,p}(\Omega)$, and let $A,B\in\AA$ with  ${\mathcal L}^d((A\cap \partial  B)\cup (B\cap\partial  A))=0$. Then 
\begin{equation}\label{claim5}
\Phi'(u,A,B)=\Phi''(u,A,B)= F(u,A,B)+\Phi_\Delta(u, A\cap B);
\end{equation}
that is, the sequence $\Phi_k(\cdot, A,B) $ $\Gamma$-converges to $\Phi(\cdot, A,B) $, where
\begin{equation}\label{claim55}
 \Phi(u,A,B):=
\int_{A\times B} f(x,y,u(x),u(y))\dxy +\int_{A\cap B} g(x,\nabla u(x))\dx
\end{equation}
for $u\in W^{1,p}(\Omega)$.
\end{theorem}

We note that the hypothesis ${\mathcal L}^d((A\cap \partial  B)\cup (B\cap\partial  A))=0$ is always satisfied if $A=B$. In particular it holds for $A=B=\Omega$, thus obtaining the representation \eqref{effe-intro} for the $\Gamma$-limit of $\Phi_k$ in \eqref{effe-k-intro}.

\begin{proof}[Proof of Theorem \rm\ref{thglo}] We first note that for all $v\in W^{1,p}(\Omega)$, $k\in\mathbb N$, and $\delta>0$ we have 
$$
\Phi_k(v,A,B)= \widecheck F^\delta_k(v,A,B)+\Phi^\delta_k(v,A,B)
$$
(see \eqref{triug}). Taking the $\Gamma$-liminf at $u$ and using the continuity Theorem \ref{contconvthm}, we have  
$$
\Phi'(u,A,B)= \widecheck F^\delta(u,A,B)+\Phi'_\delta(u,A,B)\ge \widecheck F^\delta(u,A,B)+\Phi'_\delta(u,A\cap B,A\cap B),
$$
so that, letting $\delta\to 0$ we have
\begin{equation}\nonumber
\Phi'(u,A,B)\ge  F(u,A,B)+\Phi'_\Delta(u, A\cap B)\ge  F(u,A,B)+\Phi_\Delta(u, A\cap B).
\end{equation}

To obtain claim \eqref{claim5} it is enough to prove the converse inequality 
\begin{equation}\label{convine}
\Phi''(u,A,B)\le  F(u,A,B)+\Phi_\Delta(u, A\cap B).
\end{equation}
To that end, we fix $\eta>0$. We use the notation $A_\eta=\{x\in\Omega: {\rm dist}(x,A)<\eta\}$, and set
$$
C_\eta= (A\cap B_\eta)\cup (B\cap A_\eta).
$$
For every $0<\delta<\eta$ let $u_k$ be a sequence weakly converging to $u$ in $W^{1,p}(\Omega)$ such that
\begin{eqnarray}\limsup_{k\to+\infty} \Phi^\delta_k(u_k,C_\eta,C_\eta)=
\Phi''_\delta(u,C_\eta,C_\eta).
\end{eqnarray}
Since we have
$$
(A\times B)\cap \Delta_\delta\subset (A\cap B_\eta)\times (B\cap A_\eta)\subset C_\eta\times C_\eta,
$$
we deduce that
\begin{eqnarray*}
\Phi_k(u_k,A,B)&= &\widecheck F^\delta_k(u_k,A,B)+\Phi^\delta_k(u_k,A,B)\\
&\le&\widecheck F^\delta_k(u_k,A,B)+\Phi^\delta_k(u_k,C_\eta,C_\eta),
\end{eqnarray*}
and then, by Theorem \ref{contconvthm},
\begin{eqnarray*}
\Phi''(u,A,B)\le\widecheck F^\delta(u,A,B)+\Phi''_\delta(u,C_\eta,C_\eta).
\end{eqnarray*}
Taking the limit as $\delta\to0$ and using Corollary \ref{coro}, we get
\begin{eqnarray}\label{distd}
\Phi''(u,A,B)\le F(u,A,B)+\Phi_\Delta(u,C_\eta).
\end{eqnarray}
We observe that 
$C_{\eta}\searrow (A\cap\overline B)\cup (B\cap\overline A)$ and that $\mathcal L^d((A\cap\overline B))\cup (B\cap\overline A)\setminus (A\cap B))=0$ by our hypothesis on the boundaries of $A$ and $B$. Therefore, the integral representation Theorem \ref{mainth}
implies that $\Phi_\Delta(u,C_{\eta})\to \Phi_\Delta(u,A\cap B)$, and we obtain \eqref{convine} thanks to \eqref{distd}. \end{proof}

\begin{corollary}\label{corcon} Let $w\in W^{1,\infty}(\Omega)$ and let $W^{1,p}_w(\Omega):=\{u\in W^{1,p}(\Omega): u-w\in W_0^{1,p}(\Omega)\}$. If the hypotheses of Theorem {\rm\ref{thglo}} are satisfied and $\Phi(u):=\Phi(u,\Omega,\Omega)$, then 
\begin{equation}\label{conv-min}
\lim_{k\to+\infty} \inf\{ \Phi_k(u): u\in W^{1,p}_w(\Omega)\}=\min\{ \Phi(u):u\in W^{1,p}_w(\Omega)\}.\end{equation}
Moreover, if $u_k$ is such that 
$$
\Phi_k(u_k)\le\inf\{ \Phi_k(u): u-w\in W^{1,p}_0(\Omega)\}+ o(1)
$$
as $k\to+\infty$ then, up to subsequences, $u_k$ converges weakly in $W^{1,p}(\Omega)$ to a solution to the minimum problem on the right-hand side of \eqref{conv-min}.
Finally, if in addition $\Phi(w)\le \Phi(u)$ for all $u\in W^{1,p}_w(\Omega)$ then for all sequences $\e_k>0$ with $\e_k\to0$ as $k\to+\infty$, we have 
\begin{eqnarray}\label{conv-min-0}\nonumber
&\displaystyle\lim_{k\to+\infty} \inf\{ \Phi_k(u): u\in W^{1,p}(\Omega), u(x)=w(x) \hbox{ \rm if dist}(x,\partial\Omega)\le \e_k\}\\
&\displaystyle=\min\{ \Phi(u):u\in W^{1,p}_w(\Omega)\}=\Phi(w).
\end{eqnarray}
\end{corollary}

\begin{proof} We define the functionals 
$$
\Phi^w_k(u)=\begin{cases}\Phi_k(u) &\hbox{ if } u\in W^{1,p}_w(\Omega)\\
+\infty &\hbox{ otherwise,}
\end{cases}
\qquad
\Phi^w(u)=\begin{cases}\Phi(u) &\hbox{ if } u\in W^{1,p}_w(\Omega)\\
+\infty &\hbox{ otherwise.}
\end{cases}$$
We will prove the $\Gamma$-convergence of $\Phi^w_k$ to $\Phi^w$, for which it is sufficient to show that 
for every $u\in W^{1,p}(\Omega)$ we have 
\begin{equation}\label{glivo}
\bigl(\Gamma\hbox{-}\limsup_{k\to+\infty}\Phi^w_k\bigr)(u)\le \Phi^w(u),
\end{equation}
since the liminf inequality follows from the $\Gamma$-convergence of $\Phi_k$ to $\Phi$.
The convergence of minima and minimizing sequences will then follow from the equi-coer\-cive\-ness of $\Phi^w_k$.

From the integral form of $\Phi$, the continuity properties of $\widecheck F$ (Proposition \ref{contFcheck}), and the continuity of $g$ in the second variable, we have that $\Phi^w$ is strongly continuous in $W^{1,p}_w(\Omega)$.
Hence, it suffices to show \eqref{glivo} holds for all $u\in W^{1,\infty}(\Omega)$ such that $u-w\in W^{1,p}_0(\Omega)$.

We now fix $u\in W^{1,\infty}(\Omega)$ and $\sigma>0$. We claim that there exits a compact subset $K$ of $\Omega$ such that 
\begin{equation}\label{boundK}
\limsup_{k\to+\infty}\Phi_k(u,\Omega\setminus K, \Omega\setminus K)<\sigma.
\end{equation}
Indeed,  by \eqref{doubleboundtf}--\eqref{cibi} and \eqref{usual}, for every $A\in\AA$ we have
\begin{eqnarray*}\nonumber
\Phi_k(u,A,A)\le \int_A\big((c_1+c_b){\rm Lip}(u,\Omega)^p+ a(x)\big)\dx + c_\Delta\Bigl(\int_{A} (\alpha(x)+\|u\|^q_{L^\infty(\Omega)})\dx\Bigr)^2,
\end{eqnarray*}
where ${\rm Lip}(u,\Omega)$ is the Lipschitz constant for $u$ on $\Omega$. This inequality proves the claim.

We fix $K$ such that \eqref{boundK} holds. Let now $v_k$ be a recovery sequence for $\Phi''_\delta(u,\Omega, \Omega)$ and let $u_k$ be given by Lemma \ref{lemmasub-0} with $A=\Omega$, $K\subset A'\subset\subset \Omega$,
and $B=\Omega\setminus K$, and $w_k=u$. We then have
$$
\limsup_{k\to+\infty}\Phi_k^\delta(u_k, \Omega, \Omega)\le (1+\sigma)(\Phi''_\delta(u, \Omega, \Omega)+\sigma) + \sigma.
$$
On the other hand, by \eqref{contconvcheck-1} we have
$$
\lim_{k\to+\infty}\widecheck F^\delta_k(u_k, \Omega, \Omega)=\widecheck F_\delta(u, \Omega, \Omega),
$$
so that, adding the two inequalities term by term, we obtain
$$
\limsup_{k\to+\infty}\Phi_k(u_k)\le (1+\sigma)(\Phi(u)+\sigma) + \sigma,
$$
which gives 
$$
\bigl(\Gamma\hbox{-}\limsup_{k\to+\infty}\Phi^w_k\bigr)(u)\le (1+\sigma)(\Phi(u)+\sigma) + \sigma,
$$
and hence \eqref{glivo} follows by letting $\sigma\to0$.

The last claim of the theorem can be proved likewise, choosing $u=w$ in the application of Lemma \ref{lemmasub-0} above, and noting that the corresponding $u_k$ satisfies $u_k=w$ on $\Omega\setminus A''$ for a suitable $A''$ with $A'\subset\subset A''\subset\Omega$.
This implies that 
\begin{eqnarray*}
&\displaystyle \limsup_{k\to+\infty}\ \inf\{ \Phi_k(u): u\in W^{1,p}(\Omega), u(x)=w(x) \hbox{ \rm if dist}(x,\partial\Omega)\le \e_k\}\\
&\displaystyle \le \limsup_{k\to+\infty}\Phi_k(u_k)\le (1+\sigma)(\Phi(w)+\sigma) + \sigma,
\end{eqnarray*}
and the claim letting   $\sigma\to 0$.\end{proof}

\section{Examples}

\subsection{A separation of scales effect}\label{sepsca}

We now consider a prototypical example showing a limit local energy density in which two contributions appear, one originating from the relaxation of a local integral and one produced by the concentration of a convolution energy. In this case the two contributions are decoupled. The proof of this fact is obtained by showing that recovery sequences can be constructed using two scales, on which different optimization arguments are used.

\smallskip

Let $\psi\in L^1(\mathbb R^d)$ with $\psi\ge 0$, $\psi(z)=0$ if $|z|\ge1$, and $\int_{\mathbb R^d} \psi(z)dz=1$.
We also set $\psi_k(z)= k^d\psi(kz)$.
 
Let $g_0\colon\mathbb R^d\to[0,+\infty)$ be a continuous function, and suppose that constants $0<c_0\le c_1$ and $a_0\ge 0$ exist such that
\begin{equation}\label{usual_0}
c_0|\xi|^p\le g_0(\xi)\le c_1|\xi|^p+ a_0
\end{equation}
for all  $\xi\in \R^d$.

\begin{theorem}
The functionals $\Phi_k$ defined  on $W^{1,p}(\Omega)$  by 
\begin{equation}\label{effe-k_0}
\Phi_k(u):=\int_{\Omega\times\Omega} \psi_k(x-y){|u(x)-u(y)|^p\over|x-y|^p} \dxy+\int_\Omega g_0(\nabla u(x))\,dx
\end{equation}
$\Gamma$-converge as $k\to+\infty$, with respect to the weak topology of $W^{1,p}(\Omega)$, to the functional $\Phi$ defined by
$$
\Phi(u):=\int_\Omega f_0(\nabla u(x))\dx + \int_\Omega g^{**}_0(\nabla u(x))\dx,
$$
where $g_0^{**}$ denotes the convex envelope of $g_0$ and $f_0\colon\mathbb R^d\to [0,+\infty)$ is the convex function defined by
\begin{equation}
f_0(\xi)= \int_{\mathbb R^d} \psi(z){|\xi\cdot z|^p\over |z|^p}dz.
\end{equation}
\end{theorem}

\begin{proof} We observe that $\Phi_k$ is of the form \eqref{effe-k}, where
$$
f_k(x,y,s,t)=\psi_k(x-y){|s-t|^p\over|x-y|^p}\quad\hbox{ and } \quad g_k(x,\xi)= g_0(\xi).
$$
Note that the functions $f_k$ trivially satisfy conditions \eqref{fkgrowth}--\eqref{fkcontt}, while \eqref{convextf}--\eqref{growthbk} hold with  $\widetilde f_k(x,y,\tau)=\psi_k(x-y){|\tau|^p\over|x-y|^p}$ and $b_k(x,y)=\psi_k(x-y)$. Hence, \eqref{cibi} is satisfied with $c_b=1$, and \eqref{weakconvbk} holds with  $b_\Delta=2^{-d/2}$ and $b=0$. Moreover, the functions $g_k$ satisfy hypotheses \eqref{usual}.

To prove the $\Gamma$-liminf inequality it is enough to show that
\begin{eqnarray}\label{liminf_1}
\liminf_{k\to+\infty}\int_{\Omega\times\Omega} \psi_k(x-y){|u_k(x)-u_k(y)|^p\over|x-y|^p}\dxy \ge \int_\Omega f_0(\nabla u(x))\dx,\\ \label{liminf_2}
\liminf_{k\to+\infty}\int_{\Omega} g_0(\nabla u_k(x))\dx\ge \int_\Omega g_0^{**}(\nabla u(x))\dx\,
\end{eqnarray}
for all $u_k$ weakly converging to $u$ in $W^{1,p}(\Omega)$.

Inequality \eqref{liminf_2} follows from the inequality $g_0\ge g_0^{**}$ and the weak lower semicontinuity of the functional $v\mapsto \int_\Omega g_0^{**}(\nabla v(x))\dx$ in $W^{1,p}(\Omega)$.
Inequality \eqref{liminf_1} could be achieved by using the computation of the related $\Gamma$-limit with respect to the $L^p$ convergence in \cite{BBM,AABPT}. Here we can give a direct proof, which is considerably simpler since the functions $u_k$ weakly converge in $W^{1,p}(\Omega)$. In this case,
setting
$$
\Omega_n=\bigr\{x\in \Omega: {\rm dist}\,(x,\mathbb R^d\setminus \Omega)>\tfrac1n \bigl\},
$$
using the change of variables $z=k(x-y)$, we can write
\begin{eqnarray}\label{liminf_3}\nonumber
&\displaystyle\int_{\Omega\times\Omega} \psi_k(x-y){|u_k(x)-u_k(y)|^p\over|x-y|^p}\dxy\\ \nonumber
&\displaystyle\ge \int_{\Omega_k}\int_{B_1(0)} \psi(z)k^p{|u_k(x+{1\over k}z)-u_k(x)|^p\over |z|^p}dz\dx\\ &\displaystyle\nonumber
=\int_{\Omega_k}\int_{B_1(0)} \psi(z)\Bigl|\Bigl(\int_0^1\nabla u_k(x+\tfrac{t}kz)\,dt\Bigr)\cdot {z\over |z|}\Bigr|^pdz\dx
\\ &\displaystyle
=\int_{B_1(0)}\psi(z) \int_{\Omega_k}\Bigl|\Bigl(\int_0^1\nabla u_k(x+\tfrac{t}kz)\,dt\Bigr)\cdot {z\over |z|}\Bigr|^pdx\,dz.
\end{eqnarray}
Since for all $z\in B_1(0)$ and for all $\Omega'\subset\subset\Omega$
$$
\int_0^1\nabla u_k(x+\tfrac{t}kz)\,dt \rightharpoonup \nabla u(x)
$$
weakly in $L^p(\Omega';\mathbb R^d)$ with respect to the variable $x$, also using Fatou's Lemma we obtain
\begin{eqnarray*}
\liminf_{k\to+\infty}\int_{B_1(0)}\psi(z) \int_{\Omega_k}\Bigl|\Bigl(\int_0^1\nabla u_k(x+\tfrac{t}kz)\,dt\Bigr)\cdot {z\over |z|}\Bigr|^pdx\,dz\\
\ge \int_{B_1(0)}\psi(z) \int_{\Omega'}\Bigl|\nabla u(x)\cdot {z\over |z|}\Bigr|^pdx\,dz
= \int_{\Omega'} f_0(\nabla u(x))\dx.
\end{eqnarray*}
Letting $\Omega'$ tend to $\Omega$, from this inequality and \eqref{liminf_3} we obtain \eqref{liminf_1}.

We now show the limsup inequality; that is, that for all $u\in W^{1,p}(\Omega)$ and $\eta>0$ there exist $u_k$ weakly converging to $u$ in $W^{1,p}(\Omega)$ such that
\begin{eqnarray}\label{limsup_1}
\limsup_{k\to+\infty}\int_{\Omega\times\Omega} \psi_k(x-y){|u_k(x)-u_k(y)|^p\over|x-y|^p}\dxy \le \int_\Omega f_0(\nabla u(x))\dx +\eta,\\ \label{limsup_2}
\limsup_{k\to+\infty}\int_{\Omega} g_0(\nabla u_k(x))\dx\le \int_\Omega g_0^{**}(\nabla u(x))\dx+\eta.
\end{eqnarray}

We start with the simplest case $u(x)=\xi\cdot x$. We note that
$$
g^{**}_0(\xi)=\inf\Bigl\{\int_{(0,1)^d} g_0(\xi+\nabla v(y))\dy : v\in C^\infty_c((0,1)^d)\Bigr\}
$$
(see e.g.~\cite[Section 6.2]{BDF}). Given $\eta>0$ we fix $v\in C^\infty_c((0,1)^d)$ such that
\begin{equation}\label{def-v}
\int_{(0,1)^d} g_0(\xi+\nabla v(y))\dy\le g^{**}_0(\xi)+\eta.
\end{equation}

For every $\e>0$ we define $u_\e(x)= \xi\cdot x+ \e v\bigl({x\over\e}\bigr)$. Note that 
\begin{equation}\label{limsup-13}
\limsup_{\e\to 0} \int_\Omega g_0(\nabla u_\e(x))\dx\le (g^{**}_0(\xi)+\eta)\mathcal L^d(\Omega)
\end{equation}
(see e.g.~\cite[Section 2.1]{BDF}), so that \eqref{limsup_2} holds with $u_k=u_{\e_k}$ for any sequence $\e_k\to0$. If we choose $\e_k<\!<{1\over k}$, then such a sequence also satisfies \eqref{limsup_1}. To check this, we introduce a second parameter $\zeta_k$ with 
\begin{equation}\label{doublein}
\e_k<\!<\zeta_k<\!<{1\over k}.
\end{equation}
Using the Lipschitz continuity of $v$ we infer that there exists a constant $L$ such that
$$
\Bigl|\xi\cdot {x-y\over|x-y|}+ \e_k {v\bigl({x\over\e_k}\bigr) -
v\bigl({y\over\e_k}\bigr)\over|x-y|}\Bigr|\le L.
$$
Therefore, for all $\sigma\in(0,1)$, letting $\tau=1-\sigma$, we can write
\begin{eqnarray*} &\displaystyle
\int_{\Omega\times\Omega} \psi_k(x-y){|u_{\e_k}(x)-u_{\e_k}(y)|^p\over|x-y|^p}\dxy\\
&\displaystyle=
\int_{\Omega\times\Omega} \psi_k(x-y)\Bigl|\xi\cdot {x-y\over|x-y|}+ \e_k {v\bigl({x\over\e_k}\bigr) -
v\bigl({y\over\e_k}\bigr)\over|x-y|}\Bigr|^p\dxy\\
&\displaystyle\le 
\int_{\Delta_{\zeta_k}} \hskip-.3cm\psi_k(x-y)L^p\dxy+\int_{(\Omega\times\Omega)\setminus\Delta_{\zeta_k}} \hskip-1.2cm\psi_k(x-y)\Bigl({1\over \sigma^{p-1}}\Bigl|\xi\cdot {x-y\over|x-y|}\Bigr|^p+ {2^p\e_k^p\over\tau^{p-1}}{ \|v\|^p_\infty\over\zeta_k^p}\Bigr)\dxy\\
&\displaystyle\le 
L^p{\mathcal L}^d(\Omega)\int_{B_{k\zeta_k}(0)}\hskip-.5cm \psi(z)dz+{1\over \sigma^{p-1}}\int_{\Omega} \int_{B_1(0)} \hskip-.3cm\psi(z)\Bigl|\xi\cdot {z\over|z|}\Bigr|^pdz\dx
+ {\e_k^p\over \zeta_k^p}{2^p \|v\|^p_\infty\over\tau^{p-1}}({\mathcal L}^d(\Omega))^2.
\end{eqnarray*}
Letting $k\to+\infty$ and using \eqref{doublein}, we obtain 
$$
\limsup_{k\to+\infty}\int_{\Omega\times\Omega} \psi_k(x-y){|u_{\e_k}(x)-u_{\e_k}(y)|^p\over|x-y|^p}\dxy
\le {1\over \sigma^{p-1}}\int_{\Omega}f_0( \xi){\mathcal L}^d(\Omega),
$$
which proves \eqref{limsup_1} by letting $\sigma\to1$.

If $u$ is a piecewise-affine function in $\mathbb R^d$, then $u$ is Lipschitz continuous and there exist a finite number of open simplexes $\Omega_j$ such that $\Omega\subset\bigcup_j \Omega_j$ up to a null set, such that 
$$
u(x)= \xi_j\cdot x+c_j \hbox{ for every } x\in\Omega_j.
$$
We then choose $\e_k<\!<{1\over k}$ as above and define
$$
\Omega^k_j=\bigcup\Bigl\{ \e_k i+(0,\e_k)^d: i\in\mathbb Z^d, {\rm dist} (\e_k i, \mathbb R^d\setminus \Omega_j)>{2\over k}\Bigr\},
$$
and for all $j$ set
$$
u_k(x)=\begin{cases}
u(x)+ \e_k v_j\bigl({x\over\e_j}\bigr) &\hbox{ if } x\in \Omega^k_j\\
u(x) & \hbox{ if } x\in\Omega_j\setminus \Omega^k_j,
\end{cases}
$$
where $v_j$ is defined as in \eqref{def-v} with $\xi_j$ in the place of $\xi$. 
We can then repeat the previous computations in order to prove \eqref{limsup_1} and \eqref{limsup_2}. 
The proof of the latter is simply obtained arguing as above in the sets $\Omega^k_j$ and using the Lipschitz continuity of $u$ in the remaining part of $\Omega$.

As for \eqref{limsup_1}, we subdivide $\Omega\times\Omega$ in sets of the form $\Omega_i\times \Omega_j$.
If $i=j$ the computation is exactly the same as in the first part of the proof since therein only the Lipschitz constant of $v_i$ and its $L^\infty$ norm are used. If $i\neq j$ the energy is estimated as
\begin{eqnarray*} &\displaystyle
\int_{\Omega_i\times \Omega_{j}}\hskip-.5cm \psi_k(x-y){|u(x)-u(y)|^p\over|x-y|^p}\dxy
\le L^p\mathcal L^d \bigl(\bigl\{x\in \Omega_i: {\rm dist} (x, \mathbb R^d\setminus \Omega_i)<\tfrac1k\bigr\}\bigr)\int_{B_1(0)}\hskip-.3cm\psi(z)dz,
\end{eqnarray*}
which is infinitesimal as $k\to+\infty$.

Finally, a standard argument using the density of piecewise-affine functions in $W^{1,p}(\Omega)$ allows us to conclude the limsup inequality.
\end{proof}

\subsection{Interaction with homogenization}\label{sechom}
We now consider an example of concentration of the non-local term in the presence of an underlying periodic geometry. We suppose that both in the local and non-local terms have an oscillating behaviour, with  a periodic dependence of the energy densities on the space variables. 
If the corresponding period and the scale of concentration are the same, 
optimal sequences show an interaction between oscillation and concentration, leading to a combined homogenization of the local and non-local terms. The resulting formula for the limit integrand generalizes both the classical homogenization formula for integral functionals \cite{DM,BDF} and that for convolution-type energies \cite{AABPT,BP}.
\smallskip

Let $f_0:\mathbb R^d\times \mathbb R^d\times \mathbb R^d\times \mathbb R\to[0,+\infty)$ be a Carath\'eodory function with the following properties.

(Periodicity) For all $z\in \mathbb R^d$ and $\tau\in \mathbb R$ we have
\begin{equation}
f_0(x+e_j,y+e_j,z,\tau)=f_0(x,y,z,\tau)
\end{equation}
for all $j\in\{1,\ldots, d\}$, where $e_j$ are the vectors of the canonical orthonormal basis of $\mathbb R^d$.

(Growth) There exist $\psi$ as in Section \ref{sepsca} and $C_0,C_1>0$ such that
\begin{equation}
C_0 \psi(z){\tau^p\over|z|^p}\le f_0(x,y,z,\tau)\le C_1 \psi(z){\tau^p\over|z|^p}
\end{equation}
for all  $x,y,z\in \mathbb R^d$ and $\tau\in \mathbb R$.

\smallskip
Let $g_0:\mathbb R^d\times \mathbb R\to[0,+\infty)$ be a Carath\'eodory function with the following properties.

(Periodicity) For all $x,\xi\in \mathbb R^d$ we have
\begin{equation}
g_0(x+e_j,\xi)=g_0(x,\xi) 
\end{equation}
for all $j\in\{1,\ldots, d\}$.

(Growth) There exist $c_0,c_1>0$  and $a_0\ge 0$ such that
\begin{equation}
c_0 |\xi|^p\le g_0(x,\xi)\le c_1 |\xi|^p+a_0
\end{equation}
for all  $x,y,z\in \mathbb R^d$ and $\tau\in \mathbb R$

\begin{theorem} The functionals $\Phi_\e$ defined on $W^{1,p}(\Omega)$ by
$$
\Phi_\e(u)=
{1\over\e^d} \int_{\Omega\times \Omega} f_0\Bigl({x\over\e}, {y\over\e}, {x-y\over\e},{u(x)-u(y)\over\e}\Bigr)\dxy+ \int_\Omega g_0\Bigl({x\over\e},\nabla u(x)\Bigl)\dx
$$
 $\Gamma$-converge as $\e\to0$, with respect to the weak convergence in $W^{1,p}(\Omega)$, to the functional $\Phi_{\rm hom}$ defined by
$$
\Phi_{\rm hom}(u)= \int_\Omega g_{\rm hom}(\nabla u(x))\dx,
$$
where $g_{\hom}:\mathbb R^d\to[0,+\infty)$ is given by the formula
\begin{eqnarray}\label{homform}\nonumber
&&\hskip-.7cm g_{\hom}(\xi)=\lim_{T\to+\infty} {1\over T^d} \min\Bigl\{\int_{Q_T\times Q_T} f_0(x,y,x-y,u(x)-u(y))\dxy + \int_{Q_T}g_0(x,\nabla u(x))\dx :\\
&&\hskip3cm u(x)-\xi\cdot x\in W^{1,p}_0(Q_T) \Bigl\},
\end{eqnarray}
where $Q_T=(0,T)^d$.
\end{theorem} 

\begin{proof}  To prove the theorem we fix an infinitesimal sequence of parameters $\e_k>0$ and show that the corresponding $\Phi_{\e_k}$ $\Gamma$-converge to $\Phi_{\rm hom}$.

In order to apply the results of the previous sections, we define
\begin{equation}
f_k (x,y,s,t)={1\over \e_k^d} f_0\Bigl({x\over\e_k}, {y\over\e_k}, {x-y\over\e_k},{s-t\over\e_k}\Bigr)
\end{equation}
for all  $x,y\in \mathbb R^d$ and $s,t\in \mathbb R$.
Note that $f_k$ trivially satisfies \eqref{fkgrowth}--\eqref{fkcontt} and \eqref{invariance}, while \eqref{convextf}--\eqref{growthbk} are satisfied taking $\widetilde f_k(x,y,\tau)= {b_k(x,y)\over |x-y|^p} |\tau|^p$, where
$$
b_k(x,y)={C_1\over\e_k^d}\psi\Bigl({x-y\over\e_k}\Bigr).
$$
We also define
\begin{equation}
g_k (x,\xi)=g_0\Bigl({x\over\e_k},\xi\Bigr)
\end{equation}
for all  $x,\xi \in \mathbb R^d$. Note that $g_k$ satisfies hypothesis \eqref{usual} with $a(x)=a_0$.

By \eqref{effek} and  \eqref{defGk} the functionals $\Phi_k$ defined in \eqref{ficappa} satisfy
\begin{equation}
\Phi_k(u,A,A)= {1\over\e_k^d} \int_{A\times A} f_0\Bigl({x\over\e_k}, {y\over\e_k}, {x-y\over\e_k},{u(x)-u(y)\over\e_k}\Bigr)\dxy+ \int_A g_0\Bigl({x\over\e_k},\nabla u(x)\Bigl)\dx
\end{equation} 
for all $A\in\AA$.

First note that the function $f$ in Lemma \ref{effelim} is identically $0$. Hence, by Theorem \ref{thglo} there exist a subsequence (still denoted by $\e_k$) and a function $g$ satisfying conditions \eqref{usual} such that for all $A\in\AA$ the sequence $\Phi_k(\cdot,A,A)$ $\Gamma$-converges to the functional $G(\cdot,A)$ given by
\begin{equation}
G(u,A)=\int_A g(x,\nabla u)\dx.
\end{equation}
Thanks to the Urysohn property of $\Gamma$-convergence, is it sufficient to show that $g(x,\xi)= g_{\rm hom} (\xi)$ for almost all $x$ and all $\xi\in\mathbb R^d$. 

Note that, using the same argument as in \cite[Proposition 6.1]{AABPT} we obtain that $g(x,\xi)$ is independent of $x$; that is, $g(x,\xi)=g(\xi)$. Since $G(\cdot,A)$ is lower semicontinuous with respect to the weak convergence in $W^{1,p}(\Omega)$ the function $g$ is convex, so that
$g(\xi)\le \int_Q g(\xi+\nabla u(x))\dx$  for all $u$ with $u(x)-\xi\cdot x\in W^{1,p}_0(Q)$ and $Q:=(0,1)^d$.
Since it is not restrictive to suppose that $Q\subset \Omega$ we can write
\begin{equation}\label{charcon}\nonumber
g(\xi)=\min\{ G(u,Q): u(x)-\xi\cdot x\in W^{1,p}_0(Q) \}.
\end{equation}
By Corollary \ref{corcon} we then have \begin{eqnarray}\label{concon}&&
g(\xi)=\lim_{k\to+\infty} 
\inf\{ \Phi_k(u,Q,Q): u(x)-\xi\cdot x\in W^{1,p}_0(Q) \}.
\end{eqnarray}

Let $T_k={1\over\e_k}$. By a change of variables and setting $U(x)={1\over\e_k} u(\e_k x)$ we have 
\begin{eqnarray*}
&&{1\over\e_k^d} \int_{Q\times Q} f_0\Bigl({x\over\e_k}, {y\over\e_k}, {x-y\over\e_k},{u(x)-u(y)\over\e_k}\Bigr)\dxy
\\
&=&{1\over T_k^d}\int_{Q_{T_k}\times Q_{T_k}} f_0 (x,y, x-y,U(x)-U(y))\dxy,
\end{eqnarray*}
and 
$$\int_Q g_0\Bigl({x\over\e_k},\nabla u(x)\Bigl)\dx\\
={1\over T_k^d}\int_{Q_T}g_0(x,\nabla U(x))\dx,
$$
so that we obtain
\begin{eqnarray*}
&&\inf\{ \Phi_k(u,Q,Q): u(x)-\xi\cdot x\in W^{1,p}_0(Q) \}\\
&&={1\over T_k^d} \min\Bigl\{\int_{Q_{T_k}\times Q_{T_k}} f_0(x,y,x-y,u(x)-u(y))\dxy + \int_{Q_{T_k}}g(x,\nabla u(x))\dx :\\
&& \hskip3cm u\in W^{1,p}(Q_{T_k}), u(x)-\xi\cdot x\in W^{1,p}_0(Q_{T_k})\Bigl\}.
\end{eqnarray*}
From this equality and \eqref{concon} we obtain the claim if we prove that this limit is independent of $\e_k$.

By \eqref{conv-min-0} and the change of variables above  we also have 
\begin{eqnarray}\label{concon-2}\nonumber
g(\xi)&=&\lim_{k\to+\infty} 
\inf\{ \Phi_k(u,Q,Q): u(x)-\xi\cdot x \hbox{ if dist}(x,\partial Q)\le\e_k \}\\ \nonumber
&=&{1\over T_k^d} \min\Bigl\{\int_{Q_{T_k}\times Q_{T_k}} f_0(x,y,x-y,u(x)-u(y))\dxy + \int_{Q_{T_k}}g(x,\nabla u(x))\dx :\\
&& \hskip3cm u\in W^{1,p}(Q_{T_k}), u(x)=\xi\cdot x \hbox{ \rm if dist}(x, \partial Q_{T_k})\le 1\Bigl\}.
\end{eqnarray}
The existence of the limit
\begin{eqnarray}\label{homform-2}\nonumber
&&\lim_{T\to+\infty} {1\over T^d} \min\Bigl\{\int_{Q_T\times Q_T} f_0(x,y,x-y,u(x)-u(y))\dxy + \int_{Q_T}g(x,\nabla u(x))\dx :\\
&& \hskip4cm u\in W^{1,p}(Q_T), u(x)=\xi\cdot x \hbox{ \rm if dist}(x, \partial Q_T)\le 1\Bigl\}
\end{eqnarray}
 is proved in \cite[Proposition 6.2]{AABPT} when $g_0$ is not present and the function $f$ therein is given by $f(x,z,\tau)= f_0(x,x-z,z,\tau)$. The same arguments can be used in the general case considered here. 
 This shows that also the limit in \eqref{homform} exists and concludes the proof.
\end{proof} 

\bigskip

\noindent \textsc{Acknowledgements.}
 This paper is based on work supported by the National Research Project (PRIN  2017BTM7SN) 
 "Variational Methods for Stationary and Evolution Problems with Singularities and 
 Interfaces", funded by the Italian Ministry of University and Research. 
The authors are members of the Gruppo Nazionale per 
l'Analisi Matematica, la Probabilit\`a e le loro Applicazioni (GNAMPA) of the 
Istituto Nazionale di Alta Matematica (INdAM).

\end{document}